\documentclass[11pt]{article}

\usepackage[margin=1in]{geometry}
\usepackage[T1]{fontenc}
\usepackage{lmodern}          
\usepackage{microtype}        
\usepackage{amsmath,amssymb,amsfonts,amsthm}
\usepackage{mathtools}
\usepackage{bm}
\usepackage{siunitx}
\usepackage{booktabs}
\usepackage{physics}          
\usepackage{xcolor}
\usepackage{hyperref}
\usepackage[capitalize,nameinlink]{cleveref}
\usepackage{tikz}
\usetikzlibrary{arrows.meta,calc,positioning}
\usepackage{algorithm}
\usepackage[noend]{algpseudocode}
\usepackage{enumerate}
\usepackage{authblk}

\usepackage{todonotes}

\newtheorem{theorem}{Theorem}
\newtheorem{lemma}[theorem]{Lemma}
\newtheorem{corollary}[theorem]{Corollary}
\newtheorem{proposition}[theorem]{Proposition}
\newtheorem{remark}[theorem]{Remark}

\newcommand{\bbR}{\mathbb{R}}

\newcommand{\calG}{\mathcal{G}}
\newcommand{\calV}{\mathcal{V}}
\newcommand{\calE}{\mathcal{E}}

\newcommand{\Vin}{\mathcal{V}_{\mathrm{in}}}
\newcommand{\Vout}{\mathcal{V}_{\mathrm{out}}}
\newcommand{\Vfix}{\mathcal{V}_{\mathrm{fix}}}
\newcommand{\Vctrl}{\mathcal{V}_{\mathrm{ctrl}}}
\newcommand{\Vint}{\mathcal{V}_{\mathrm{int}}}

\newcommand{\matB}{\mathbf{B}}     
\newcommand{\matC}{\mathbf{C}}     
\newcommand{\matL}{\mathbf{L}}     

\newcommand{\vecc}{\mathbf{c}}     %
\newcommand{\vecd}{\mathbf{d}}     %
\newcommand{\vecu}{\mathbf{u}}     
\newcommand{\vecq}{\mathbf{q}}     
\newcommand{\vecx}{\mathbf{x}}     %
\newcommand{\vecPhi}{\bm{\Phi}}    
\newcommand{\vecg}{\mathbf{g}}     

\newcommand{\uzero}{\vecu_0}
\newcommand{\Umap}{\mathbf{U}_g}
\newcommand{\qzero}{\vecq_0}
\newcommand{\Qmap}{\mathbf{Q}_g}
\newcommand{\Phizero}{\vecPhi_0}
\newcommand{\Pmap}{\mathbf{P}_g}

\newcommand{\Efix}{\mathbf{E}_{\mathrm{fix}}}
\newcommand{\Ectrl}{\mathbf{E}_{\mathrm{ctrl}}}
\newcommand{\Eint}{\mathbf{E}_{\mathrm{int}}}

\DeclareMathOperator{\diag}{diag}
\newcommand{\ones}{\mathbf{1}}

\title{Optimal Boundary Control of Diffusion on Graphs via Linear Programming \thanks{This work is partially supported by the Office of Naval Research (ONR) under Award NO: N00014-24-
1-2147, NSF grant DMS-2408877, and the Air Force Office of Scientific Research (AFOSR) under Award NO: FA9550-25-1-0231.}}

	\author[1]{Harbir Antil\thanks{Email: \texttt{hantil@gmu.edu}}}	
	\author[2]{Rainald L\"ohner\thanks{Email: \texttt{rlohner@gmu.edu}}}
        \author[3]{Felipe P\'erez\thanks{Email: \texttt{fperezsi@gmu.edu}}}
	\affil[1,3]{\small{Department of Mathematical Sciences and the Center for Mathematics and Artificial Intelligence (CMAI), George Mason University, Fairfax, VA 22030, USA.}}
	\affil[2]{\small{Department of Physics and the Computational Fluid Dynamics Center, George Mason University, Fairfax, VA 22030, USA.}}

\date{}

\begin{document}
\maketitle


\vspace{-1cm}

\begin{abstract}
We propose a linear programming (LP) framework for steady-state diffusion and flux optimization on geometric networks. 
The state variable satisfies a discrete diffusion law on a weighted, oriented graph, where conductances are scaled by edge lengths to preserve geometric fidelity. 
Boundary potentials act as controls that drive interior fluxes according to a linear network Laplacian. 
The optimization problem enforces physically meaningful sign and flux-cap constraints at all boundary edges, derived directly from a gradient bound $|\nabla u|\le\phi_{\max}$. 
This yields a finite-dimensional LP whose feasible set is polyhedral, and whose boundedness and solvability follow from simple geometric or algebraic conditions on the network data.

We prove that under the absence of negative recession directions—automatically satisfied in the presence of finite box bounds, flux caps, or sign restrictions—the LP admits a global minimizer. 
Several sufficient conditions guaranteeing boundedness of the feasible region are identified, covering both full-rank and rank-deficient flux maps. 
The analysis connects classical results such as the Minkowski–Weyl decomposition, Hoffman’s bound, and the fundamental theorem of linear programming with modern network-based diffusion modeling.

Two large-scale examples illustrate the framework: 
(i) A typical large stadium in a major modern city, 
which forms a single connected component
with relatively uniform corridor widths, and a 
(ii) A complex street network emanating from a large, 
historical city center, which forms a multi-component system.
 
In both cases, the LP formulation yields stable and physically consistent flux fields that satisfy sign and conservation constraints to within double-precision tolerance. 
The approach generalizes naturally to other diffusion-controlled phenomena such as thermal transport, chemical dispersion, or pedestrian flow in complex geometric environments.
\end{abstract}

\section{Introduction}

Many planning, transport, and diffusion processes on geometric networks can be formulated as steady-state boundary-control problems. 
The goal is to prescribe or optimize boundary potentials (Dirichlet data) so that the resulting flux field achieves a desired steady-state configuration inside the network. 
Examples include heat transport through composite materials, contaminant propagation in pipe systems, and pedestrian or vehicular flow in constrained infrastructures such as stadiums, terminals, or urban layouts. 
In all such systems, the governing equations are linear, yet the control variables—boundary conditions—are subject to physical and geometric constraints that must be respected. 
Optimizing these boundary controls thus leads naturally to \emph{PDE-constrained optimization} problems posed on graphs.

\vspace{0.3em}
This work develops a general \emph{linear programming (LP)} formulation for such diffusion-control problems on weighted, oriented graphs. 
Each edge $e=(t(e),h(e))$ (tail and head, respectively) is assigned a length $L_e$ and a cross-sectional weight $k_e$, defining a conductance 
\[
c_e = k_e / L_e 
\]
consistent with Fick’s or Fourier’s law. 
The discrete flux $q_e = -c_e(u_{h(e)} - u_{t(e)})$ links vertex potentials $u_{t(e)}$ and $u_{h(e)}$, and vertex-level balance enforces discrete conservation. 
Boundary vertices are divided into inflow and outflow sets $(\mathcal{V}_{\text{in}}, \mathcal{V}_{\text{out}})$, and the boundary potentials $\vecg$ serve as controls. 
The interior state is determined uniquely by solving a discrete Laplace equation, producing an \emph{affine control-to-state mapping}
\[
(\vecu, \vecq, \vecPhi)
  = (\vecu_0, \vecq_0, \vecPhi_0)
  + (\mathbf{U}, \mathbf{Q}, \mathbf{P})\, \vecg.
\]
This affine reduction converts the continuous diffusion control problem into a finite-dimensional LP with explicit geometric and physical inequality constraints.

\paragraph{Modeling.}
Three structural elements distinguish the present framework:
\begin{enumerate}[(i)]\itemsep0pt
  \item \textbf{Edge-level sign constraints.}
  Backflow is eliminated by enforcing orientation-aware sign restrictions on boundary fluxes, ensuring outwardness and physical admissibility.
  \item \textbf{Flux caps from gradient bounds.}
  A maximum gradient constraint $|\partial_s u|\le\phi_{\max}$ implies per-edge flux caps $|q_e|\le k_e\phi_{\max}$, yielding linear inequalities in $\vecg$.
  \item \textbf{Affine elimination of the state.}
  The linear diffusion operator permits exact elimination of interior variables, resulting in a polyhedral feasible set and a globally convex optimization problem.
\end{enumerate}

\paragraph{Mathematical framework.}
The feasible region of the LP
\[
P = \{\vecg \in \mathbb{R}^n : A \vecg \le b,\; \vecg_{\min} \le \vecg \le \vecg_{\max}\}
\]
is a convex polyhedron defined by affine flux relations and capacity bounds. 
We show that the objective functional, being linear in $\vecg$, attains a global minimum whenever no negative recession direction $\vecd \in \operatorname{rec}(P)$ satisfies $\vecc^\top \vecd < 0$ (see~\eqref{eq:cdef} for definition of $\vecc$). 
This condition, verified using the Minkowski–Weyl decomposition of polyhedra, guarantees existence of an optimal solution without assuming boundedness \emph{a priori}. 
Practical sufficient conditions for boundedness arise naturally:
finite box bounds, full-rank flux maps with two-sided caps, or rank-deficient flux maps regularized by sign constraints. 
These results connect geometric arguments from convex analysis and linear programming—such as Hoffman’s error bound \cite{AJHoffman_1952a}, 
Rockafellar’s theory of convex sets \cite{rockafellar1970convex}, and Schrijver’s foundational work on polyhedral theory \cite{ASchrijver_1986a}—
to diffusion control on networks.

\paragraph{Algorithmic formulation.}
After eliminating the interior state, the optimization problem reduces to
\[
\min_{\vecg}\; -(\vecc^\star)^\top\vecg
\quad\text{s.t.}\quad
A_{\mathrm{cap}}\vecg \le b_{\mathrm{cap}},\;
A_{\mathrm{edge}}\vecg \le b_{\mathrm{edge}},\;
\vecg_{\min} \le \vecg \le \vecg_{\max},
\]
where $(A_{\mathrm{edge}}, b_{\mathrm{edge}})$ encode per-edge sign constraints with a small slack $\varepsilon \ge 0$, 
and $(A_{\mathrm{cap}}, b_{\mathrm{cap}})$ enforce $|q_e| \le \phi_{\max}k_e$.
This LP can be solved directly using standard solvers, and the affine control–state map allows post-solution reconstruction of all nodal potentials and fluxes. 
The LP structure ensures global optimality, simple feasibility checks, and numerical robustness without requiring penalty parameters or nonlinear relaxations.

\paragraph{Connections to prior work.} 
Diffusion and flow on graphs are classical subjects \cite{chung1997spectral,doyle2000random,boyd2004convex}, 
and linear programming has long been used for network-flow optimization \cite{ahuja1993network,DBertsimas_JNTsitsiklis_1997a,ziegler1995lectures}. 
However, traditional approaches typically optimize edge flows or capacities directly. 
Our formulation differs in two key respects:  
(i) fluxes arise from a diffusive (Ohmic) relation rather than from independent capacity constraints; and  
(ii) the optimization acts on \emph{boundary potentials}, not on edge injections, while enforcing sign constraints that preserve physical outwardness. 
On the analysis side, solvability and boundedness follow from convex-analytic principles 
(recession cones, Hoffman bounds) rather than combinatorial flow arguments.  
In the context of evacuation modeling, this framework is consistent with continuum potential-flow theories (e.g., \cite{hughes2002continuum}) but operates natively on networks, 
which is essential when only a line-graph abstraction is available. We also refer to \cite{MSchmidt_FMHante_2023a}, and references therein, for gas transportation on networks using PDE constrained optimization frameworks.

\paragraph{Contributions.}
\begin{itemize}\itemsep0pt
  \item A unified LP formulation for boundary-control diffusion problems with edge-level sign and flux-cap constraints.
  \item A geometric proof of existence and boundedness using recession cones and Minkowski–Weyl decomposition.
  \item Simple sufficient conditions for automatic boundedness under physical modeling assumptions.
  \item Numerical validation on two real-world networks confirming conservation, sign correctness, and geometric consistency.
\end{itemize}

\section{Model on a metric graph}
Let $\calG=(\mathcal V,\mathcal E)$ be a finite connected graph. Each edge 
$e\in\mathcal E$ is a 1D segment of length $L_e>0$ with constant conductivity $k_e>0$,
oriented from its tail $t(e)$ to head $h(e)$. 
We are given disjoint boundary sets: inflows $\Vin$ and outflows $\Vout$. The user chooses a fixed subset $\Vfix\subseteq\Vin$ with prescribed temperatures $u_{\mathrm{fix}}$. All remaining boundary nodes are controls:
\begin{equation*}
  \Vctrl := (\Vin\cup\Vout)\setminus\Vfix,
  \qquad \vecg := \vecu\big|_{\Vctrl}\in\bbR^{|\Vctrl|}.
\end{equation*}
Special cases: $\Vfix=\emptyset$ (all boundary are controls) or $\Vctrl=\emptyset$ (no optimization).
We denote nodal temperatures by $u_v$ for $v\in\mathcal V$ and edgewise temperatures by $u_e(x)$ for $x\in(0,L_e)$
with $x=0$ at $t(e)$ and $x=L_e$ at $h(e)$.

On edge $e$, the steady diffusion equation holds
\begin{equation*}
  -\dv{}{x}\!\left(k_e \, \dv{u}{x}\right)=0
\end{equation*}
The solution is affine: $u_e(x)=u_{t(e)} + \frac{x}{L_e}\,(u_{h(e)}-u_{t(e)})$.
The (signed) constant flux along $e$ (positive in the $t(e)\to h(e)$ direction) is
\begin{equation*}
q_e \;:=\; -\,k_e\,\frac{d u_e}{dx} \;=\; -\,\frac{k_e}{L_e}\,\big(u_{h(e)}-u_{t(e)}\big) 
= -c_e \,\big(u_{h(e)}-u_{t(e)}\big) ,
\end{equation*}
with edge conductance $c_e := k_e/L_e$. 

At node $v\in\calV$, the nodal flux balance (inflow minus outflow) is
\begin{equation*}
  \Phi_v \;:=\; \sum_{e \in \mathcal{E}:\,h(e)=v} q_e \;-\; \sum_{e\in \mathcal{E}:\,t(e)=v} q_e,
  \qquad \sum_{v\in\calV}\Phi_v=0.
\end{equation*}
Interior nodes enforce conservation: $\Phi_v=0$ for all $v\notin \Vin\cup\Vout$.

\section{Discrete operators and shapes}
Index nodes by $1, \dots, n_V$ and orient each edge with a tail and head. We use an \emph{edge-by-node} incidence matrix $\matB\in\bbR^{n_E\times n_V}$:
\begin{equation}\label{eq:coincidence_matrix}
  (\matB)_{e,v} = \begin{cases}
  -1, & v=t(e),\\
  +1, & v=h(e),\\
  0, & \text{otherwise.}
  \end{cases}
\end{equation}
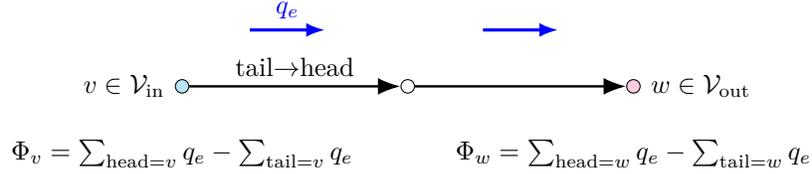
\begin{figure}[htb]
\centering
\begin{tikzpicture}[scale=1.0]
  \node[circle,draw,fill=cyan!25,inner sep=1.8pt,label=left:{\small $v\in\Vin$}] (vin) at (0,0) {};
  \node[circle,draw,inner sep=1.8pt] (mid) at (3.0,0) {};
  \node[circle,draw,fill=magenta!25,inner sep=1.8pt,label=right:{\small $w\in\Vout$}] (vout) at (6.0,0) {};
  \draw[-{Latex[length=3mm]},line width=0.9pt] (vin) -- node[above]{\small tail$\to$head} (mid);
  \draw[-{Latex[length=3mm]},line width=0.9pt] (mid) -- (vout);
  \draw[-{Latex[length=2.5mm]},line width=1.2pt,blue] (0.9,0.75) -- (1.9,0.75) node[midway,above] {\small $q_e$};
  \draw[-{Latex[length=2.5mm]},line width=1.2pt,blue] (4.0,0.75) -- (5.0,0.75);
  \node at (0,-0.9) {\small $\Phi_v=\sum_{\mathrm{head}=v}q_e-\sum_{\mathrm{tail}=v}q_e$};
  \node at (6.0,-0.9) {\small $\Phi_w=\sum_{\mathrm{head}=w}q_e-\sum_{\mathrm{tail}=w}q_e$};
\end{tikzpicture}
\caption{Orientation and signs. Incidence $\matB$ is edge-by-node with $-1$ at tail and $+1$ at head. Then $\vecq=-\matC\matB\vecu$ and $\vecPhi=\matB^\top\vecq$ implement the definitions compactly.}
\end{figure}

Let $\matC=\diag(c_e)\in\bbR^{n_E\times n_E}$. For a nodal vector $\vecu\in\bbR^{n_V}$,
\begin{equation}\label{eq:discLap}
\begin{aligned}
  &\text{edge fluxes:}\quad \vecq \;=\; -\,\matC\,\matB\,\vecu \in\bbR^{n_E}, \\  
  &\text{nodal balances:}\quad \vecPhi \;=\; \matB^\top \vecq \;=\; -\,\matB^\top \matC \matB\,\vecu \;=:\; -\matL\,\vecu,
\end{aligned}  
\end{equation}
where $\matL:=\matB^\top\matC\matB\in\bbR^{n_V\times n_V}$ is the weighted graph Laplacian.
To summarize, we have: $\matB \in \mathbb{R}^{n_E{\times}n_V}$, $\matC \in \mathbb{R}^{n_E{\times}n_E}$, $\matL \in \mathbb{R}^{n_V{\times}n_V}$, $\vecu \in \mathbb{R}^{n_V{\times}1}$, $\vecq \in \mathbb{R}^{n_E{\times}1}$, $\vecPhi \in \mathbb{R}^{n_V{\times}1}$.

\paragraph{Embedding matrices.}
For any node index set $S\subset\{1,\dots,n_V\}$, the embedding $\mathbf{E}_S\in\bbR^{n_V\times |S|}$ places a subvector at the global indices of $S$. We will use $\Efix$ (for $\Vfix$), $\Ectrl$ (for $\Vctrl$), and $\Eint$ (for $\Vint:=\{1,\dots,n_V\}\setminus(\Vfix\cup\Vctrl)$).

\section{Mixed Dirichlet boundary and affine maps}
Write the nodal state as
\begin{equation}
  \vecu \;=\; \Eint\,\vecu_{\mathrm{int}} \;+\; \Efix\,\vecu_{\mathrm{fix}} \;+\; \Ectrl\,\vecg .
\end{equation}
Interior conservation ($\vecPhi|_{\Vint}=0$) gives
\begin{equation}
  0 \;=\; -\Eint^\top \matL \vecu \;=\; -\Big(\underbrace{\Eint^\top\matL\Eint}_{\matL_{ii}}\vecu_{\mathrm{int}} + \underbrace{\Eint^\top\matL\Efix}_{\matL_{i,\mathrm{fix}}}\vecu_{\mathrm{fix}} + \underbrace{\Eint^\top\matL\Ectrl}_{\matL_{i,\mathrm{ctrl}}}\vecg \Big).
\end{equation}
Assuming each connected component has at least one Dirichlet node (fixed or controlled), $\matL_{ii}\succ 0$ (see below) and
\begin{equation}
  \vecu_{\mathrm{int}}(\vecg) \;=\; \underbrace{-\matL_{ii}^{-1}\matL_{i,\mathrm{fix}}\vecu_{\mathrm{fix}}}_{A_0} \;+\; \underbrace{-\matL_{ii}^{-1}\matL_{i,\mathrm{ctrl}}}_{A_1}\,\vecg .
\end{equation}
Reconstruct the full state
\begin{equation}\label{eq:ug}
  \vecu(\vecg) \;=\; \uzero \;+\; \Umap \vecg, 
  \quad \uzero := \Eint A_0 + \Efix \vecu_{\mathrm{fix}}, 
  \quad \Umap := \Eint A_1 + \Ectrl .
\end{equation}
Then
\begin{equation}\label{eq:qg}
\begin{aligned}
  \vecq(\vecg) &\;=\; \qzero + \Qmap\vecg \;:=\; -\matC\matB\,\uzero \;+\; \big(-\matC\matB\,\Umap\big)\vecg, 
  \\ 
  \vecPhi(\vecg) &\;=\; \Phizero + \Pmap\vecg \;:=\; \matB^\top \qzero \;+\; \big(\matB^\top\Qmap\big)\vecg .
\end{aligned}  
\end{equation}
For boundary blocks, define
\begin{equation}\label{eq:Phig}
\begin{aligned}
  \vecPhi_{\mathrm{in}}(\vecg)=\Phizero|_{\Vin}+K_{\mathrm{in}}\vecg,\quad K_{\mathrm{in}}:=\Pmap(\Vin,:), 
  \\
  \vecPhi_{\mathrm{out}}(\vecg)=\Phizero|_{\Vout}+K_{\mathrm{out}}\vecg,\quad K_{\mathrm{out}}:=\Pmap(\Vout,:).
\end{aligned}  
\end{equation}

We prove a few straightforward results which are critical to justify the above calculations, in-particular the invertabilty of $\matL_{ii}$. Towards this end, let $\mathbf{1}_{\mathcal G_i}\in\mathbb{R}^{n_V}$ denotes the indicator vector of the
$i$-th connected component $\mathcal G_i$ of the graph, defined by
\[
(\mathbf{1}_{\mathcal G_i})_v =
\begin{cases}
1, & \text{if vertex } v \in \mathcal G_i,\\[2mm]
0, & \text{otherwise.}
\end{cases}
\]
Each $\mathbf{1}_{\mathcal G_i}$ is constant on its component and vanishes elsewhere,
so that the nullspace of the weighted Laplacian $\matL=\matB^\top\matC\matB$
is precisely $\ker(\matL)=\mathrm{span}\{\mathbf{1}_{\mathcal G_1},\dots,
\mathbf{1}_{\mathcal G_K}\}$ which we characterize next. 
\begin{lemma}[Weighted Laplacian structure and nullspace]\label{lem:lap-kernel}
Let $\matB\in\bbR^{n_E\times n_V}$ be the edge-by-node incidence as given in \eqref{eq:coincidence_matrix}.  
Let $\matC=\diag(c_e)\succ 0$ with edge conductances $c_e>0$, and define $\matL$ as in \eqref{eq:discLap}. 
Then:
\begin{enumerate}[(a)]
  \item $\matL$ is symmetric positive semidefinite (SPSD).
  \item If the graph has $K$ connected components $\calG_1,\dots,\calG_K$, then
  \[
    \ker(\matL) \;=\; \mathrm{span}\{\mathbf{1}_{\calG_1},\dots,\mathbf{1}_{\calG_K}\},
  \]
  i.e., the kernel is spanned by the indicator vectors of the components (constant on each component, zero elsewhere).
\end{enumerate}
\end{lemma}

\begin{proof}
\emph{(a) SPSD.} Symmetry is immediate: $\matL^\top=(\matB^\top\matC\matB)^\top=\matB^\top\matC\matB=\matL$ since $\matC$ is diagonal.
For any $\vecu\in\bbR^{n_V}$,
\[
  \vecu^\top \matL \vecu \;=\; \vecu^\top \matB^\top \matC \matB \vecu
  \;=\; (\matB\vecu)^\top \matC (\matB\vecu)
  \;=\; \sum_{e\in\calE} c_e\,\big((\matB\vecu)_e\big)^2 \;\ge\; 0,
\]
so $\matL$ is positive semidefinite.

\smallskip
\noindent\emph{(b) Kernel characterization.}
We have $\vecu\in\ker(\matL)$ iff $\matL\vecu= \mathbf{0}$, i.e.,
\[
  0 \;=\; \vecu^\top \matL \vecu \;=\; \sum_{e\in\calE} c_e\,\big((\matB\vecu)_e\big)^2,
\]
which, since $c_e>0$, is equivalent to $(\matB\vecu)_e=0$ for all edges $e$.
But $(\matB\vecu)_e = \vecu_{h(e)} - \vecu_{t(e)}$ by construction, so $\vecu_{h(e)}=\vecu_{t(e)}$ for every edge.
Therefore $\vecu$ is \emph{constant on each connected component}: along any path inside a component, adjacent nodes must have equal values.
Conversely, any vector that is constant on each component satisfies $\matB\vecu= \mathbf{0}$, hence $\matL\vecu= \mathbf{0}$.
The set of such vectors is precisely $\mathrm{span}\{\mathbf{1}_{\calG_1},\dots,\mathbf{1}_{\calG_K}\}$.
\end{proof}

The next result discusses the invertability of $\matL_{ii}$ and states the affine dependence of $\vecu$, $\vecq$, and $\vecPhi$ on $\vecg$. 
\begin{theorem}[Forward solve]\label{thm:forward}
If each connected component of $\calG$ contains at least one Dirichlet node (fixed or controlled), then the interior block
\[
  \matL_{ii} \;=\; \Eint^\top \matL \Eint
\]
is symmetric positive definite and the interior system has a unique solution. Consequently
$\vecu(\vecg)$, $\vecq(\vecg)$, and $\vecPhi(\vecg)$ depend affinely on $\vecg$.
\end{theorem}

\begin{proof}
By Lemma~\ref{lem:lap-kernel}, $\ker(\matL)$ is spanned by componentwise constants.
Consider $x\in\bbR^{|\Vint|}$ with $\matL_{ii}x= \mathbf{0}$. Let $\tilde x := \Eint x$ be the vector on all nodes that equals $x$ on $\Vint$ and zero on Dirichlet nodes $D:=\Vfix\cup\Vctrl$.
Then
\[
  0 \;=\; x^\top \matL_{ii} x \;=\; (\Eint x)^\top \matL (\Eint x)
  \;=\; \sum_{e\in\calE} c_e\,\big((\matB\,\Eint x)_e\big)^2,
\]
so $\matB\,\Eint x= \mathbf{0}$ and $\Eint x$ is constant on each connected component.
But on any component that contains at least one Dirichlet node, $\Eint x$ is zero at that Dirichlet node by construction (since $\Eint$ places zeros on $D$). The only constant function that is zero at some node is the zero constant, hence $\Eint x= \mathbf{0}$ on that component. Because \emph{every} component has at least one Dirichlet node by assumption, $\Eint x= \mathbf{0}$ globally, so $x=\mathbf{0}$.
Thus $\matL_{ii}$ is injective; being symmetric PSD, it is positive definite. Uniqueness and affinity of $(\vecu,\vecq,\vecPhi)$ in $\vecg$ follow \eqref{eq:ug}, \eqref{eq:qg}, and \eqref{eq:Phig}.
\end{proof}

\begin{corollary}[Necessity of anchoring each component]\label{cor:necessity}
If some connected component contains no Dirichlet node (neither fixed nor controlled), then the interior block $\matL_{ii}$ is singular. The forward problem on that component is only determined up to an additive constant, and the overall state is not unique.
\end{corollary}

\begin{proof}
If a component has no Dirichlet node, the corresponding constant vector lies in $\ker(\matL)$ and has support entirely on that component. Its restriction to $\Vint$ is nonzero and lies in $\ker(\matL_{ii})$, so $\matL_{ii}$ is not invertible.
\end{proof}

\begin{remark}[Practical interpretation]
Dirichlet nodes (fixed \emph{or} controlled) \emph{anchor} the potential on each connected component by removing the constant-mode nullspace of $\matL$. Without at least one anchor per component, the “DC offset” is free (physically: only gradients are determined), making the linear system underdetermined and potentially causing numerical issues or unbounded directions in optimization.
\end{remark}

\section{Objective: maximize outward boundary flux}
Outward means $\vecPhi_{\mathrm{in}}\le 0$ at inflows (out of boundary into interior) and $\vecPhi_{\mathrm{out}}\ge 0$ at outflows. We maximize the net outward boundary flux
\begin{equation}
\label{eq:cdef}
  J(\vecg) \;=\; -\ones^\top \vecPhi_{\mathrm{in}}(\vecg) \;+\; \ones^\top \vecPhi_{\mathrm{out}}(\vecg)
  \;=\; \text{const} + \underbrace{\Big(\begin{bmatrix}K_{\mathrm{in}}\\K_{\mathrm{out}}\end{bmatrix}^\top \begin{bmatrix}-\ones\\ \ones\end{bmatrix}\Big)}_{=:~\vecc^\star} \vecg .
\end{equation}
Since typical optimization solvers, e.g., \texttt{linprog} minimizes, we solve the following minimization problem:
$$
\min_{\vecg} -(\vecc^\star)^\top \vecg .
$$
We are still not done specifying the optimization problem fully.

\section{No-backflow constraints}

\paragraph{Edge-level (per-edge) sign constraints.}
For each boundary endpoint $v$ and adjacent edge $e=(\mathrm{tail}\to\mathrm{head})$, outwardness implies:
\[
\begin{array}{lll}
v\in\Vin:  & t(e)=v \Rightarrow q_e \ge -\varepsilon, \quad h(e)=v \Rightarrow q_e \le +\varepsilon,\\[1mm]
v\in\Vout: & h(e)=v \Rightarrow q_e \ge -\varepsilon, \quad t(e)=v \Rightarrow q_e \le +\varepsilon.
\end{array}
\]
Collect these as $S\vecq \ge -\varepsilon$ with selector $S\in\bbR^{m\times n_E}$ (each row picks $\pm q_e$). With $\vecq(\vecg)=\qzero+\Qmap\vecg$,
\begin{equation*}
  S(\qzero+\Qmap\vecg) \ge -\varepsilon
  \quad\Longleftrightarrow\quad
  (-S\Qmap)\,\vecg \le S\qzero + \varepsilon .
\end{equation*}
Here we denote
\begin{equation}\label{eq:Aedge}
A_{\mathrm{edge}} := -S\Qmap, 
\qquad 
b_{\mathrm{edge}} := S\qzero + \varepsilon ,
\end{equation}
so that the edge sign constraints compactly read
\[
A_{\mathrm{edge}}\,\vecg \;\le\; b_{\mathrm{edge}}.
\]
Each row of $A_{\mathrm{edge}}$ corresponds to one boundary edge condition,
with the sign pattern determined by the selector matrix $S$.
Edges between same-type boundary nodes (Vin--Vin or Vout--Vout) receive two opposite constraints, forcing $q_e\simeq 0$ (no boundary shunt).

The next result states the impact of slack $\varepsilon$ on the total flux. The shunt case will be discussed in the result following this result.
\begin{lemma}[Edge rules imply node sign constraints up to slack]\label{lem:edge-to-node}
Let $v$ be a boundary node with degree $\deg(v)$ (number of incident edges). Impose the edge-level outwardness rules with slack $\varepsilon\ge 0$:
\[
\begin{array}{lll}
v\in\Vin:  & t(e)=v \Rightarrow q_e \ge -\varepsilon,\quad h(e)=v \Rightarrow q_e \le +\varepsilon,\\[1mm]
v\in\Vout: & h(e)=v \Rightarrow q_e \ge -\varepsilon,\quad t(e)=v \Rightarrow q_e \le +\varepsilon.
\end{array}
\]
Then the nodal flux balance satisfies
\[
\boxed{~
\begin{aligned}
v\in\Vin &: \quad \Phi_v \;\le\; \varepsilon\,\deg(v),\\
v\in\Vout&: \quad \Phi_v \;\ge\; -\,\varepsilon\,\deg(v).
\end{aligned}
}
\]
In particular, for $\varepsilon=0$ the node-level sign constraints $\Phi_v\le 0$ on $\Vin$ and $\Phi_v\ge 0$ on $\Vout$ follow from the edge rules.
\end{lemma}
\begin{proof}
Recall $\Phi_v=\sum_{h(e)=v} q_e - \sum_{t(e)=v} q_e$.\\
If $v\in\Vin$, each incident edge with $h(e)=v$ obeys $q_e\le +\varepsilon$ and each with $t(e)=v$ obeys $q_e\ge-\,\varepsilon$, hence
\[
\Phi_v \;=\; \sum_{h(e)=v} q_e - \sum_{t(e)=v} q_e
\;\le\; \#\{h(e)=v\}\,\varepsilon \;-\;\big(-\#\{t(e)=v\}\,\varepsilon\big)
\;=\; \varepsilon\,\deg(v).
\]
If $v\in\Vout$, then $h(e)=v\Rightarrow q_e\ge-\,\varepsilon$ and $t(e)=v\Rightarrow q_e\le+\,\varepsilon$, so
\[
\Phi_v \;\ge\; -\,\#\{h(e)=v\}\,\varepsilon \;-\;\#\{t(e)=v\}\,\varepsilon
\;=\; -\,\varepsilon\,\deg(v).
\]
Setting $\varepsilon=0$ yields $\Phi_v\le 0$ on $\Vin$ and $\Phi_v\ge 0$ on $\Vout$. \qedhere
\end{proof}

\begin{proposition}[Boundary–boundary edges are blocked]\label{prop:bb-block}
If an edge $e$ connects two boundary nodes of the \emph{same type}, i.e., $t(e),h(e)\in\Vin$ or $t(e),h(e)\in\Vout$, then the edge rules imply
\[
\boxed{\; |q_e| \;\le\; \varepsilon. \;}
\]
Thus for $\varepsilon=0$ such edges carry no flux (no boundary shunting).
\end{proposition}

\begin{proof}
For $t(e),h(e)\in\Vin$: the tail rule gives $q_e\ge -\varepsilon$ and the head rule gives $q_e\le +\varepsilon$, hence $|q_e|\le \varepsilon$. The Vout–Vout case is identical.
\end{proof}

\begin{remark}[Practical takeaway]
Lemma~\ref{lem:edge-to-node} shows that adding edge-level rules with a small slack
$\varepsilon$ \emph{automatically} enforces the desired node-level sign constraints
up to $O(\varepsilon)$, while Proposition~\ref{prop:bb-block} prevents artificial
``short-circuits'' along boundary--boundary edges. In implementation, these
constraints appear as the linear inequality
\[
A_{\mathrm{edge}}\,\vecg \le b_{\mathrm{edge}},
\quad
A_{\mathrm{edge}}=-S\Qmap,\quad
b_{\mathrm{edge}}=S\qzero+\varepsilon,
\]
where each row of $S$ encodes the sign pattern for one boundary edge.
Choosing $\varepsilon$ at or slightly above the solver tolerance
(e.g.\ $10^{-6}$ in physical units) effectively removes cosmetic backflow
without altering the physical solution.
\end{remark}

\section{Flux caps from a gradient bound}
A gradient cap $\abs{u'}\le \phi_{\max}$ on edge $e$ implies
\begin{equation*}
  \abs{q_e} = k_e\,\abs{u'} \le \phi_{\max}\,k_e = \phi_{\max}\,c_e L_e.
\end{equation*}
If raw conductivities $k_e$ are provided (as in data), take $q_{\max,e}=\phi_{\max}k_e$ (equivalently $\phi_{\max}c_eL_e$). In matrix form:
\begin{equation*}
  -\vecq_{\max} \le \qzero+\Qmap\vecg \le \vecq_{\max}
  \;\Longleftrightarrow\;
  \begin{bmatrix} \Qmap \\ -\Qmap \end{bmatrix}\vecg \le
  \begin{bmatrix} \vecq_{\max}-\qzero \\ \vecq_{\max}+\qzero \end{bmatrix}.
\end{equation*}
For compactness, we define
\[
A_{\mathrm{cap}} :=
\begin{bmatrix}
\Qmap \\[1mm] -\Qmap
\end{bmatrix},
\qquad
b_{\mathrm{cap}} :=
\begin{bmatrix}
\vecq_{\max}-\qzero \\[1mm]
\vecq_{\max}+\qzero
\end{bmatrix},
\]
so that the flux‐cap inequalities take the standard form
\[
A_{\mathrm{cap}}\,\vecg \;\le\; b_{\mathrm{cap}}.
\]

\section{Linear program}

Collecting all components, the optimization problem reads
\begin{equation}
\label{eq:LP_form}
\begin{aligned}
  \min_{\vecg \in \mathbb{R}^{|\Vctrl|}} 
  &\quad J(\vecg) := -(\vecc^\star)^\top \vecg,\\
  \text{s.t.}\quad 
  & A_{\mathrm{cap}}\,\vecg \le b_{\mathrm{cap}}, \\[1mm]
  & A_{\mathrm{edge}}\,\vecg \le b_{\mathrm{edge}}
      \qquad \text{(edge no-backflow with slack $\varepsilon$)},\\[1mm]
  & \vecg_{\min} \le \vecg \le \vecg_{\max}
      \qquad \text{(optional box bounds).}
\end{aligned}
\end{equation}
Here all constraints are affine, so the feasible set is an intersection of finitely many halfspaces and slabs. To prove this result, we assume that the objective function is bounded below.
\begin{theorem}[Convexity and existence of an optimum]
\label{thm:exist}
The feasible region of~\eqref{eq:LP_form} is a convex polyhedron, and the objective is linear;
therefore the problem is convex. 
If the feasible set is nonempty and the objective is bounded below, a global minimizer exists.
\end{theorem}
\begin{proof}
Let the decision variable be $\vecg \in \mathbb{R}^n$ and write the constraints compactly as
\[
A\vecg \le b, \qquad 
\vecg_{\min} \le \vecg \le \vecg_{\max}.
\]
The feasible set
\[
P := \{\vecg : A\vecg \le b,\;
\vecg_{\min} \le \vecg \le \vecg_{\max}\}
\]
is an intersection of finitely many closed halfspaces (and possibly slabs, i.e., intersections of two halfspaces); hence $P$ is a closed convex
polyhedron \cite{rockafellar1970convex}. 
Moreover, $J(\vecg) = \vecc^\top \vecg$ is linear and therefore convex.

Assume $P\neq\emptyset$ and $\inf_{\vecg\in P} \vecc^\top \vecg > -\infty$.
Define the recession cone
\[
\mathrm{rec}(P)
= \Bigg\{\vecd : A\vecd \le 0,\;
\begin{array}{l}
d_i \ge 0 \ \text{if } (g_{\min})_i > -\infty \text{ and } (g_{\max})_i=+\infty,\\
d_i \le 0 \ \text{if } (g_{\max})_i < +\infty \text{ and } (g_{\min})_i=-\infty,\\
d_i = 0 \ \text{if both } (g_{\min})_i,(g_{\max})_i \text{ are finite}
\end{array}
\Bigg\},
\]
with no sign restriction when both bounds are infinite.
Then the standard equivalence holds:
\[
\vecd\in \mathrm{rec}(P)
\quad\Longleftrightarrow\quad
\vecg+\tau \vecd\in P\ \ \text{for every }\vecg\in P\text{ and every }\tau\ge 0.
\]
($\Rightarrow$) From $A\vecd\le 0$ and $A\vecg\le b$,
$A(\vecg+\tau\vecd)\le b$ for all $\tau\ge 0$. The coordinatewise sign/zero rules ensure the bounds are preserved.
($\Leftarrow$) If $\vecg+\tau\vecd\in P$ for all $\tau\ge 0$, then
$A\vecg+\tau A\vecd\le b$ implies $A\vecd\le 0$ by dividing by $\tau$ and letting $\tau\to\infty$; the same “for all $\tau$” argument forces the stated sign/zero conditions on~$\vecd$.

If there existed $\vecd\in\mathrm{rec}(P)$ with $\vecc^\top \vecd<0$, then
$J(\vecg+\tau \vecd)=\vecc^\top \vecg + \tau\,\vecc^\top \vecd \to -\infty$, contradicting boundedness from below.
Hence
\[
\vecc^\top \vecd \;\ge\; 0\qquad\text{for all }\vecd\in\mathrm{rec}(P).
\]

By the Minkowski--Weyl representation (e.g., \cite[Cor.~2.7]{DBertsimas_JNTsitsiklis_1997a}),
there exist finite sets $V$ (points) and $R$ (rays) such that
$P=\operatorname{conv}(V)+\operatorname{cone}(R)$ and $\operatorname{rec}(P)=\operatorname{cone}(R)$.
Any $\vecx\in P$ can be written $\vecx=\sum_i\lambda_i v^i+\sum_j\mu_j r^j$ with
$\lambda_i\ge0$, $\sum_i\lambda_i=1$, $\mu_j\ge0$. Then
\[
\vecc^\top \vecx=\sum_i\lambda_i\,\vecc^\top v^i+\sum_j\mu_j\,\vecc^\top r^j
\ \ge\ \sum_i\lambda_i\,\vecc^\top v^i
\ \ge\ \inf_{\mathbf{v}\in\operatorname{conv}(V)} \vecc^\top \mathbf{v},
\]
since $\vecc^\top r^j\ge 0$ for all rays. The reverse inequality holds by taking $\mu\equiv 0$, so
\[
\inf_{\vecx\in P} \vecc^\top \vecx \;=\; \inf_{\mathbf{v}\in\operatorname{conv}(V)} \vecc^\top \mathbf{v}.
\]
If $V\neq\varnothing$, the feasible region $P$ has at least one vertex. Since $P$ is a closed polyhedron and the objective function $\vecc^\top \vecg$ is a continuous linear function, the minimum value is guaranteed to exist due to the boundedness-from-below assumption. The optimal value is attained at one of the vertices of the feasible region, a consequence of the fundamental theorem of linear programming.
$
    \inf_{\vecx\in P} \vecc^\top \vecx = \min_{\mathbf{v} \in V} \vecc^\top \mathbf{v}.
$
The convex hull $\operatorname{conv}(V)$ is a compact set (a bounded polytope) and the objective function is linear, so the minimum is attained at an extreme point (a vertex).

If $V=\varnothing$, the polyhedron $P$ contains no vertices (extreme points). Since $P\neq\varnothing$, this implies that $P$ is an unbounded polyhedron that contains a line. The existence of a finite optimal value, $\inf_{\vecg\in P} \vecc^\top \vecg = \alpha > -\infty$, guarantees that the optimal solution set is non-empty and closed. This optimal set is a face of the polyhedron $P$, defined by $P\cap\{\vecg: \vecc^\top \vecg=\alpha\}$. The minimum is therefore attained on this face, even though no vertices exist.
\end{proof}

Next, we prove a general result, which shows that for $\vecd \in \mbox{rec}(P)$ with $\vecc^\top \vecd \ge 0$ then $J(\vecg)$ is bounded below. In order to apply this result to our setting of Theorem~\ref{thm:exist}, we will provide sufficient conditions in Proposition~\ref{prop:boundedness} which guarantees that $\vecc^\top \vecd \ge 0$. 

We begin by stating the well-known Hoffman bound \cite{AJHoffman_1952a} and \cite[pp.~175]{MHong_JZJuo_2017a}. 
\begin{lemma}[Hoffman bound for polyhedra]
\label{lem:hoffman}
Let $\mathcal{F}=\{x\in\mathbb{R}^n: Mx\le q\}$ be nonempty, where $M\in\mathbb{R}^{p\times n}$, $q\in\mathbb{R}^p$.
For any norm $\|\cdot\|$ on $\mathbb{R}^n$, there exists a constant $H=H(M,\|\cdot\|)\in(0,\infty)$ such that
for every $x\in\mathbb{R}^n$,
\[
\mathrm{dist}(x,\mathcal{F})
\;\le\;
H\,\big\| (Mx-q)_+ \big\|_\infty,
\]
where $(\cdot)_+$ is the componentwise positive part.
\end{lemma}

\begin{lemma}[Linear growth of projection distances onto LP sublevel sets]
\label{lem:linear-growth}
Let $P=\{g:Ag\le b\}$ be nonempty and $c\in\mathbb{R}^n$.
Fix $g^0\in P$, and for each $k\in\mathbb{N}$ let
\[
P_{-k}:=\{g\in P:\ c^\top g\le -k\}
\]
be nonempty. Choose $g_k\in\arg\min\{\|g-g^0\|:\ g\in P_{-k}\}$.
Then there exists $H\in(0,\infty)$, depending only on $(A,b,c, g^0)$ and the chosen norm,
such that
\[
\|g_k-g^0\|\ \le\ H\,(1+k)\qquad\text{for all }k.
\]
\end{lemma}

\begin{proof}
Write the augmented system for $P_{-k}$ as
\[
Mx\le q_k
\quad\Longleftrightarrow\quad
\begin{bmatrix} A \\ c^\top \end{bmatrix} x \;\le\;
\begin{bmatrix} b \\ -k \end{bmatrix}.
\]
Let $H$ be a Hoffman constant for the matrix $M=[A;\,c^\top]$ with respect to the chosen norm.
Applying Hoffman’s bound (Lemma~\ref{lem:hoffman}) at $x=g^0$ gives
\[
\mathrm{dist}(g^0,P_{-k})
\;\le\;
H\,\big\| (M g^0 - q_k)_+ \big\|_\infty
=
H\,\Big\| \Big( (A g^0 - b)_+,\ \big(c^\top g^0 - (-k)\big)_+ \Big) \Big\|_\infty .
\]
Since $g^0\in P$, we have $(A g^0 - b)_+=0$, so
\[
\mathrm{dist}(g^0,P_{-k})
\;\le\;
H\,\big(c^\top g^0 + k\big)_+ 
=
\begin{cases}
0, & c^\top g^0 \le -k,\\[1mm]
H\,(c^\top g^0 + k), & c^\top g^0 > -k.
\end{cases}
\]
By definition of $g_k$ as a nearest point to $P_{-k}$,
\[
\|g_k - g^0\| \;=\; \mathrm{dist}(g^0,P_{-k})
\;\le\; H\,(c^\top g^0 + k)_+
\;\le\; H\,(1 + |c^\top g^0|)\,(1+k).
\]
Renaming the constant $H \leftarrow H\,(1+|c^\top g^0|)$ yields the claimed linear bound
$\|g_k-g^0\|\le H\,(1+k)$ for all $k$.
\end{proof}

\begin{lemma}[Boundedness of a linear functional on a polyhedron]
\label{lem:boundedness1}
Let $P := \{\, g \in \mathbb{R}^n : A g \le b \,\}$ be a nonempty polyhedron and
$\mathrm{rec}(P) := \{\, d \in \mathbb{R}^n : A d \le 0 \,\}$ its recession cone.
If every $d \in \mathrm{rec}(P)$ satisfies $c^{\top} d \ge 0$, then the linear functional
$J(g)=c^{\top}g$ is bounded below on~$P$.
\end{lemma}

\begin{proof}
Assume for contradiction that $J$ is unbounded below on $P$.
Then for each integer $k\ge1$ the sublevel set $P_{-k}:=\{g\in P:\ c^\top g\le -k\}$ is nonempty.
Fix $g^0\in P$ and, for each $k$, choose $g_k\in P_{-k}$ minimizing $\|g-g^0\|$.
Define the positive scalars
\[
t_k := -\big(c^\top g_k - c^\top g^0\big)\ \xrightarrow[k\to\infty]{}\ +\infty,
\qquad
\hat d_k := \frac{g_k-g^0}{t_k}.
\]
By Lemma~\ref{lem:linear-growth}, $\|g_k-g^0\|\le H(1+k)$, hence
\[
\|\hat d_k\| = \frac{\|g_k-g^0\|}{t_k}
\;\le\; \frac{H(1+k)}{k - c^\top g^0}
\;\le\; 2H
\quad\text{for all large }k,
\]
so $(\hat d_k)$ is bounded. By Bolzano--Weierstrass, extract a convergent subsequence
$\hat d_{k_j}\to \hat d$.

Since $A g_k\le b$ and $A g^0\le b$,
\[
A\hat d_k = \frac{A g_k - A g^0}{t_k}
\;\le\; \frac{b - A g^0}{t_k} \;\xrightarrow[k\to\infty]{}\; 0,
\]
which implies $A\hat d\le 0$; 
hence $\hat d\in\mathrm{rec}(P)$.
Moreover,
\[
c^\top \hat d_k
= \frac{c^\top g_k - c^\top g^0}{t_k}
= -1 \quad\Rightarrow\quad c^\top \hat d = -1 < 0.
\]
Thus there exists a nonzero recession direction $\hat d$ with $c^\top \hat d<0$,
contradicting the hypothesis. Therefore $J$ is bounded below on $P$.
\end{proof}

\begin{proposition}[When boundedness (or bounded-below) is automatic]
\label{prop:boundedness}
Let $P = \{\vecg \in \mathbb{R}^n : A \vecg \le b, \; \vecg_{\min} \le \vecg \le \vecg_{\max}\}$,
and suppose the flux map is affine: $\vecq(\vecg) = \qzero + \Qmap \vecg$.
Then:
\begin{enumerate}[(i)]
\item \textbf{Explicit box bounds.} If all components have finite bounds
      $\vecg_{\min} \le \vecg \le \vecg_{\max}$, then $P$ is a compact polytope (closed and bounded).
\item \textbf{Flux caps with full column rank.} If two-sided flux caps
      $|\vecq(\vecg)| \le \vecq_{\max}$ hold and $\Qmap$ has full column rank on the control subspace,
      then $\|\Qmap \vecg\|_\infty \le \|\vecq_{\max}-\qzero\|_\infty$ implies
      $\|\vecg\|_2 \le C$ for some $C<\infty$; hence $P$ is bounded.
\item \textbf{Rank-deficient case with sign constraints.} If $\Qmap$ is rank-deficient, but edge sign constraints
      eliminate every recession direction $\vecd\in\mathrm{rec}(P)$ with $\vecc^\top \vecd<0$ (no “descent” rays),
      then the linear objective is bounded below on $P$ (by Lemma~\ref{lem:boundedness1}).
      \emph{(Note: $P$ itself can remain unbounded due to neutral rays with $\vecc^\top \vecd=0$.)}
\end{enumerate}
\end{proposition}

\begin{proof}
(i) The intersection with a finite hyperrectangle is closed and bounded, hence compact.

(ii) From $-\vecq_{\max} \le \qzero + \Qmap \vecg \le \vecq_{\max}$ we obtain
$-\vecq_{\max}-\qzero \le \Qmap \vecg \le \vecq_{\max}-\qzero$, i.e.
$\|\Qmap \vecg\|_\infty \le \|\vecq_{\max}-\qzero\|_\infty$.
Let $\Qmap\in\mathbb{R}^{m\times n}$ and assume it has full \emph{column} rank, so
$\sigma_{\min}(\Qmap)>0$. Then
\[
\|\vecg\|_2 \le \frac{\|\Qmap \vecg\|_2}{\sigma_{\min}(\Qmap)}
             \le \frac{\sqrt{m}\,\|\Qmap \vecg\|_\infty}{\sigma_{\min}(\Qmap)}
             \le \frac{\sqrt{m}\,\|\vecq_{\max}-\qzero\|_\infty}{\sigma_{\min}(\Qmap)}.
\]
Thus $\vecg$ is norm-bounded and $P$ is bounded.

(iii) Without flux caps, $\mathrm{rec}(P)$ may be nontrivial. If the sign constraints (together with
$A\vecg\le b$ and the bound structure) exclude every $d\in\mathrm{rec}(P)$ with $\vecc^\top d<0$,
then by Lemma~\ref{lem:boundedness1} the objective is bounded below on $P$.
Neutral rays with $\vecc^\top d=0$ may remain, so $P$ need not be bounded.
\end{proof}

\section{Numerical experiments}
\label{sec:numerical}

The main driver is Algorithm~\ref{alg:driver} with subroutines 
given in Algorithms~\ref{alg:build}--\ref{alg:recover-verify}. 
In all examples, we set $\phi_{\max} = 1$ and $\varepsilon = 0$.

We present two representative networks of contrasting structure and connectivity: a \emph{typical large stadium in a major 
modern city}, which forms a single connected component
with relatively uniform corridor widths, and a 
\emph{complex street network emanating from a large, historical city center},
which decomposes into seven connected subgraphs with pronounced geometric heterogeneity.
The same linear programming (LP) formulation applies to both, demonstrating that
the proposed framework is robust to both dense single-component and fragmented multi-component geometries.

\paragraph{Gauge fixing for control invariance.}
When all Dirichlet values on a connected component are treated as free controls,
adding a constant to every boundary value on that component leaves edge fluxes and
nodal fluxes unchanged. The LP objective and constraints are therefore invariant
to such rigid shifts, creating a null direction in control space.
To remove this gauge freedom and improve conditioning,
we impose one scalar gauge per connected component---either by fixing one Dirichlet node
or enforcing a mean-zero constraint on its boundary controls.
In the following examples, one node per component is fixed to a prescribed potential value.

\paragraph{Flux metrics.}
We visualize two complementary quantities:
\begin{itemize}
  \item \emph{Flux intensity:} the magnitude $|q_e|$, representing per-area flux density (diffusive strength per unit width);
  \item \emph{Throughput:} the product $|q_e|A_e$, representing total transported quantity through an edge (aggregate capacity use).
\end{itemize}
Together, these measures distinguish between relative diffusion intensity and total corridor throughput.

The limit set for the flux corresponds to the empirically
observed limit of pedestrian flux intensity, which is of the
order of 1 person/meter/second for safe crowd walking 
conditions \cite{fruin1971pedestrian,predtetschenski1971personenstrome}. Thus, the results obtained
represent the upper limit of steady pedestrian flux that these
networks are able to sustain in a safe manner.

\subsection{Typical Large Stadium in a Modern City}
\label{sec:bernabeu}

We first demonstrate the LP framework on a realistic geometric network extracted via google maps for a large stadium located
in a modern city. The network only considers a partial set
of streets, no subway exits, and pedestrian motion only
(no cars, buses or trains). 
The graph $\calG=(\calV,\calE)$ encodes the concourse and corridor layout of the stadium and its immediate 
surroundings.
Vertices correspond to corridor intersections, and 
edges to walkable links (avenues, streets, passages)
with assigned cross-sectional areas.
Conductances are scaled inversely with edge length to emulate diffusive pedestrian flow
or steady-state heat conduction along corridors of varying width. 

\subsubsection{Geometry and cross-sectional area distribution}

Figure~\ref{fig:bernabeu-area} shows the reconstructed network geometry
colored by effective edge area.
Blue edges correspond to narrow passages, while yellow–red edges indicate wider corridors
or open concourse sections.
The lower panel zooms into the interior, revealing the main radial and circumferential corridors
that govern overall circulation.

\begin{figure}[h!]
\centering
\includegraphics[width=0.8\textwidth]{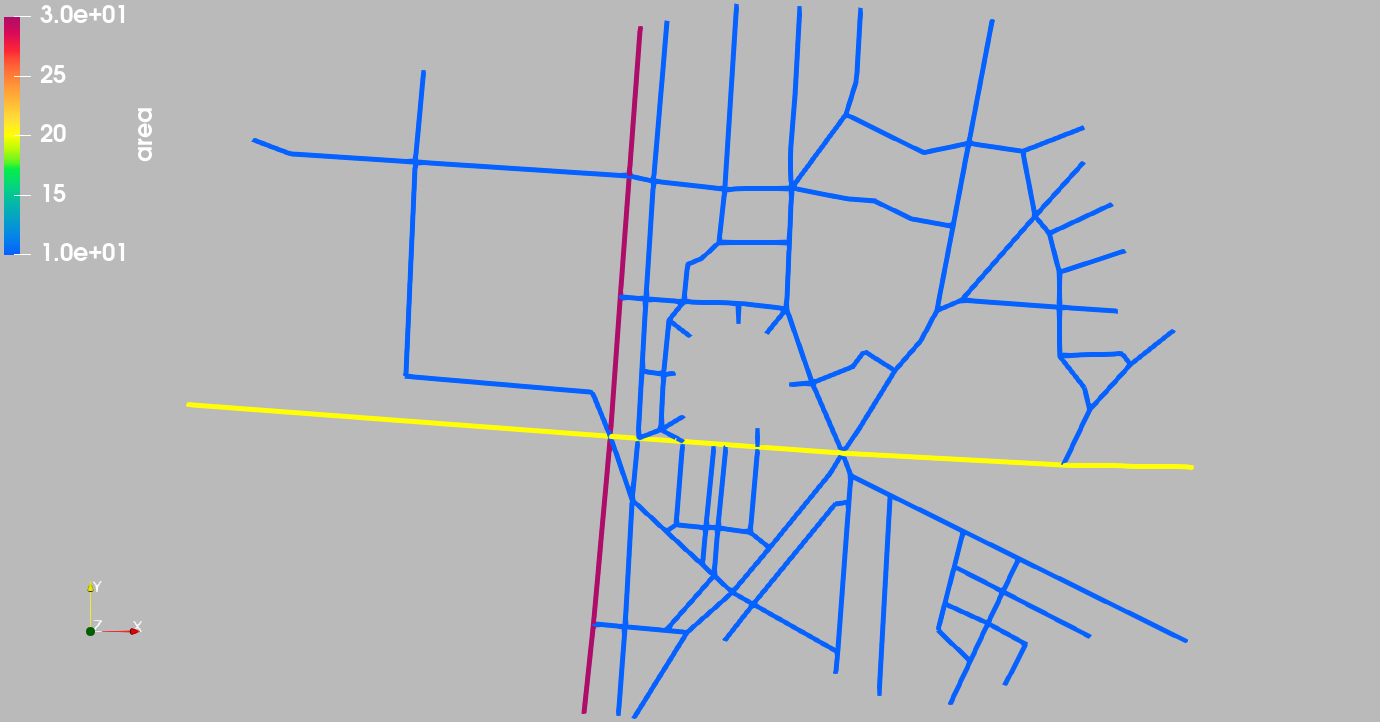}
\vspace{1em}
\includegraphics[width=0.8\textwidth]{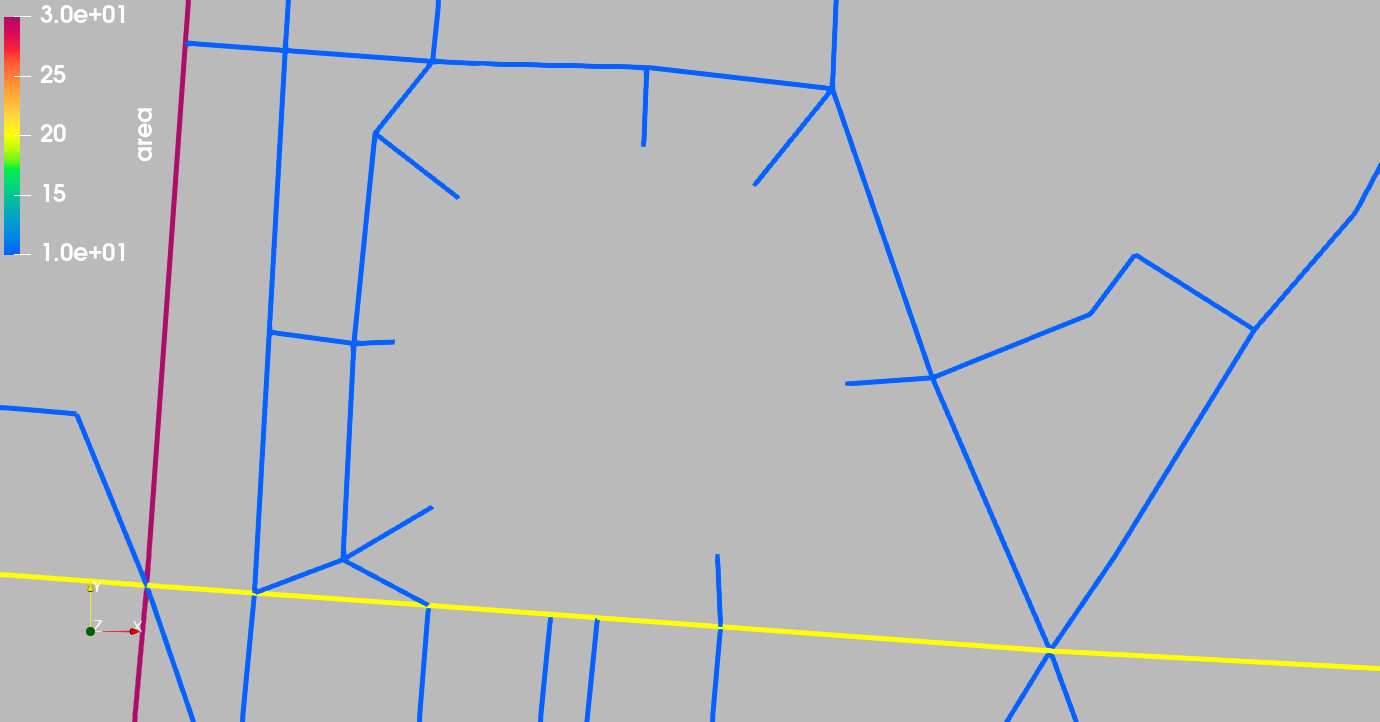}
\caption{Stadium network colored by corridor cross-sectional area.
Blue: narrow links; red: broad concourse regions.
The zoomed view highlights the dense interior structure near the stadium core.}
\label{fig:bernabeu-area}
\end{figure}

\subsubsection{Temperature (potential) distribution}

The steady-state potential field obtained from the optimized control
is shown in Figure~\ref{fig:bernabeu-temp}.
Higher potentials (red–yellow) occur near entrance boundaries,
while cooler regions (green–blue) correspond to exits or high-dissipation paths.
We fix the potential at one interior node ($113$) to a reference value $u_{113}=10$
to remove the additive nullspace of the steady-state diffusion operator.
This normalization pins the global potential field without affecting gradients or flux balance.

\begin{figure}[h!]
\centering
\includegraphics[width=0.8\textwidth]{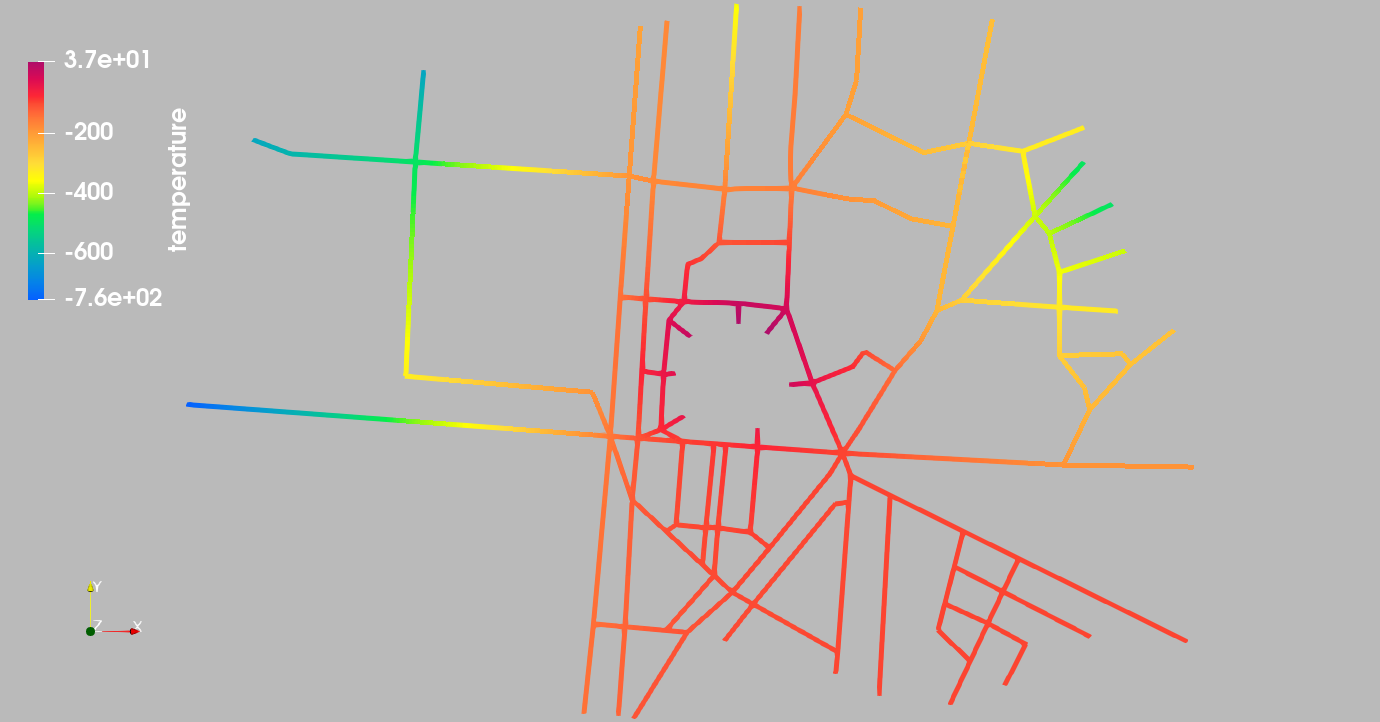}
\caption{Optimized potential field $u$ for the stadium network.
Hotter colors (red–yellow) mark entrances or high-pressure zones;
cooler tones (blue–green) mark exits and dissipative corridors.
The reference potential was fixed at node~113 ($u_{113}=10$).}
\label{fig:bernabeu-temp}
\end{figure}

\subsubsection{Flux intensity and throughput}

Figures~\ref{fig:bernabeu-flux}–\ref{fig:bernabeu-flxxarea}
show the computed flux intensity $|q_e|$ and throughput $|q_e|A_e$
across the network.
High-intensity fluxes (yellow–red) appear along the primary north–south
and east–west corridors, indicating strong directional transport
between major entry and exit points.
Flux directions remain consistent with the stadium’s
natural pedestrian flow topology.
After multiplying by area, the throughput map emphasizes
broad corridors that carry higher total flow.

\begin{figure}[h!]
\centering
\includegraphics[width=0.8\textwidth]{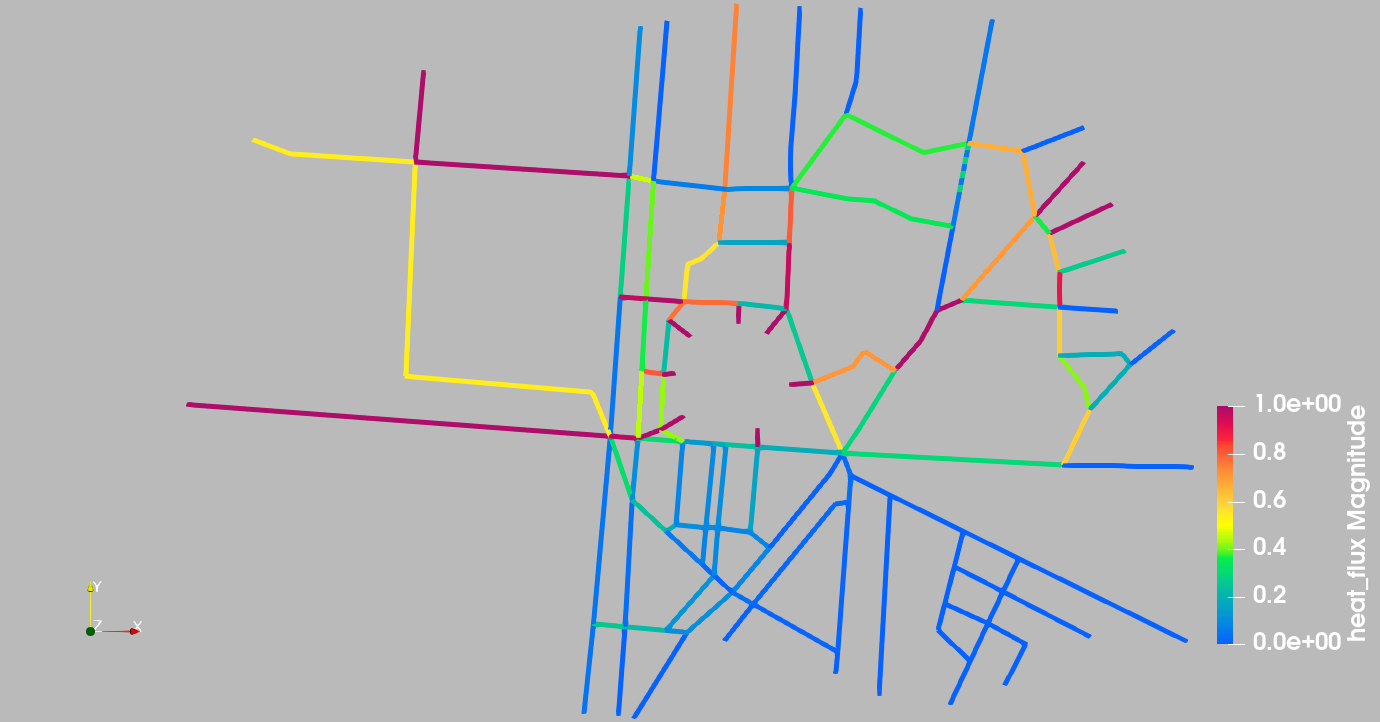}
\vspace{1em}
\includegraphics[width=0.8\textwidth]{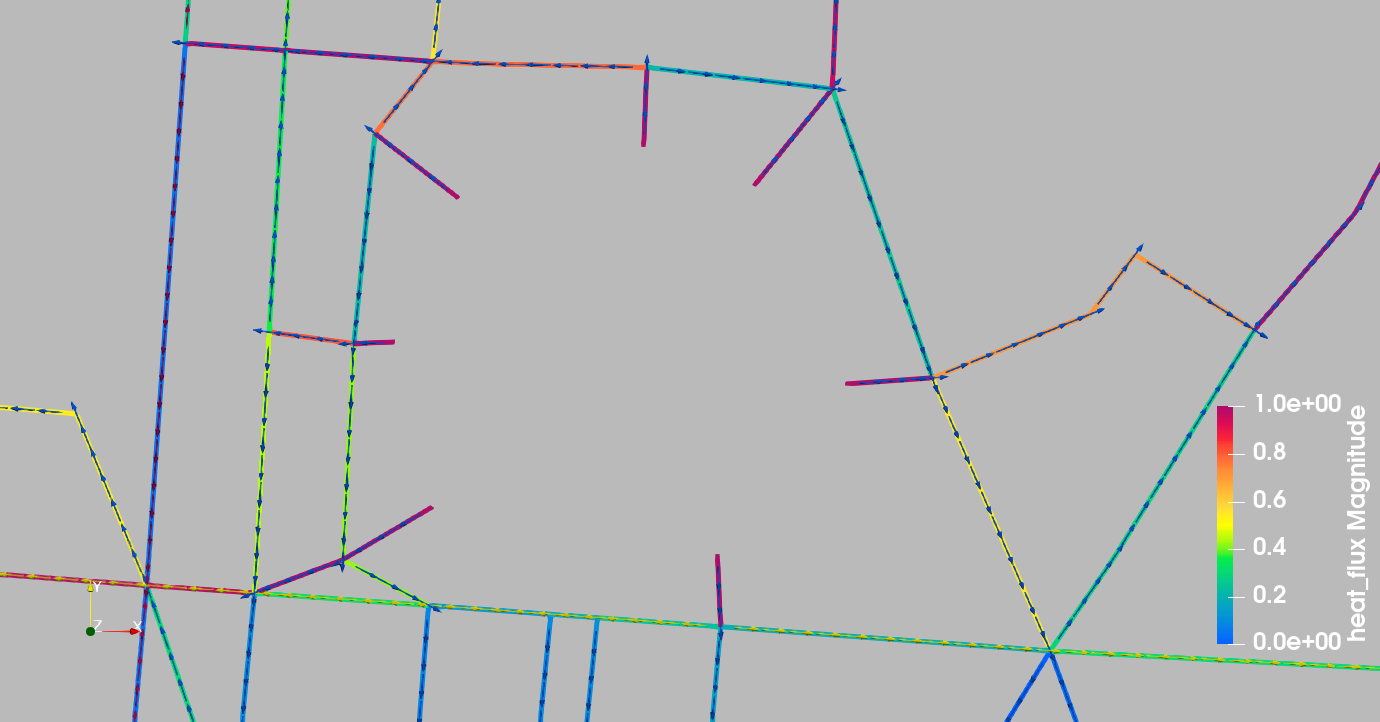}
\caption{Flux intensity $|q_e|$ across the network.
Top: global view of all corridors. Bottom: zoomed view near the stadium interior.
Color represents per-area flow strength; arrows indicate flux direction.}
\label{fig:bernabeu-flux}
\end{figure}

\begin{figure}[h!]
\centering
\includegraphics[width=0.8\textwidth]{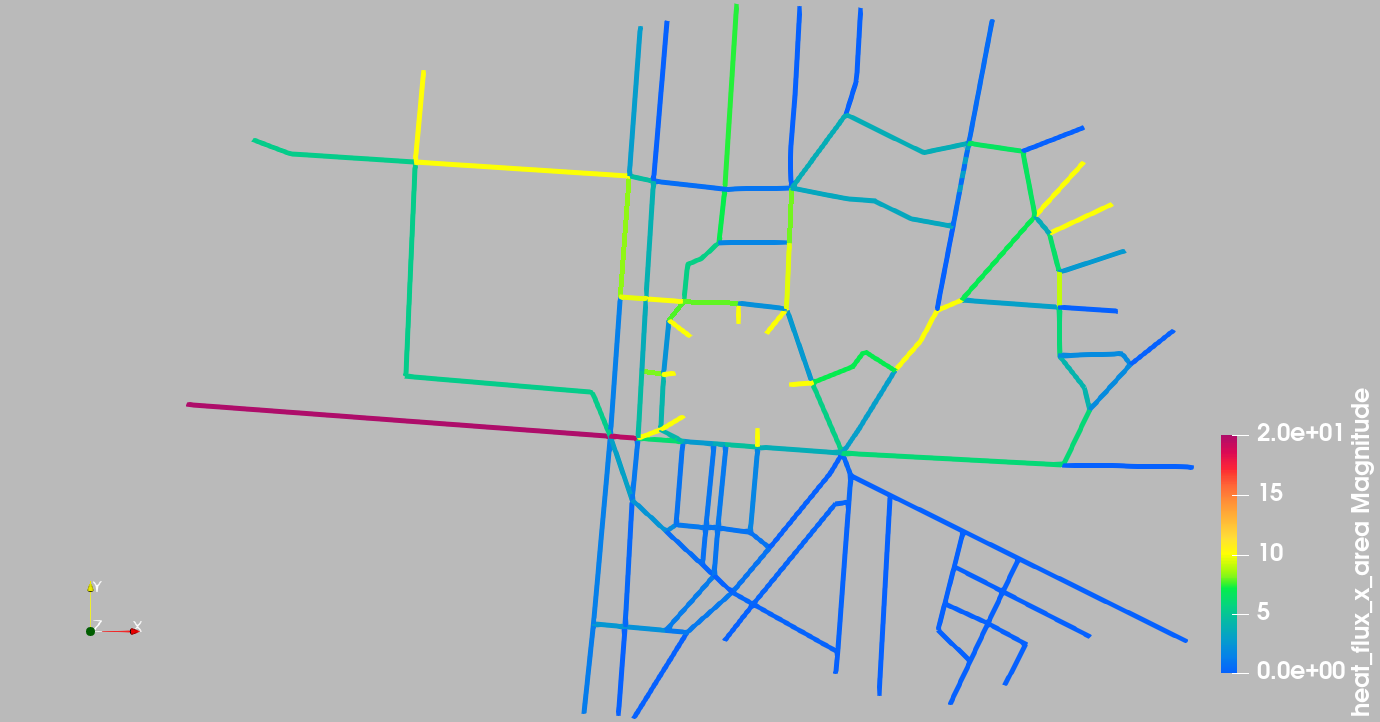}
\vspace{1em}
\includegraphics[width=0.8\textwidth]{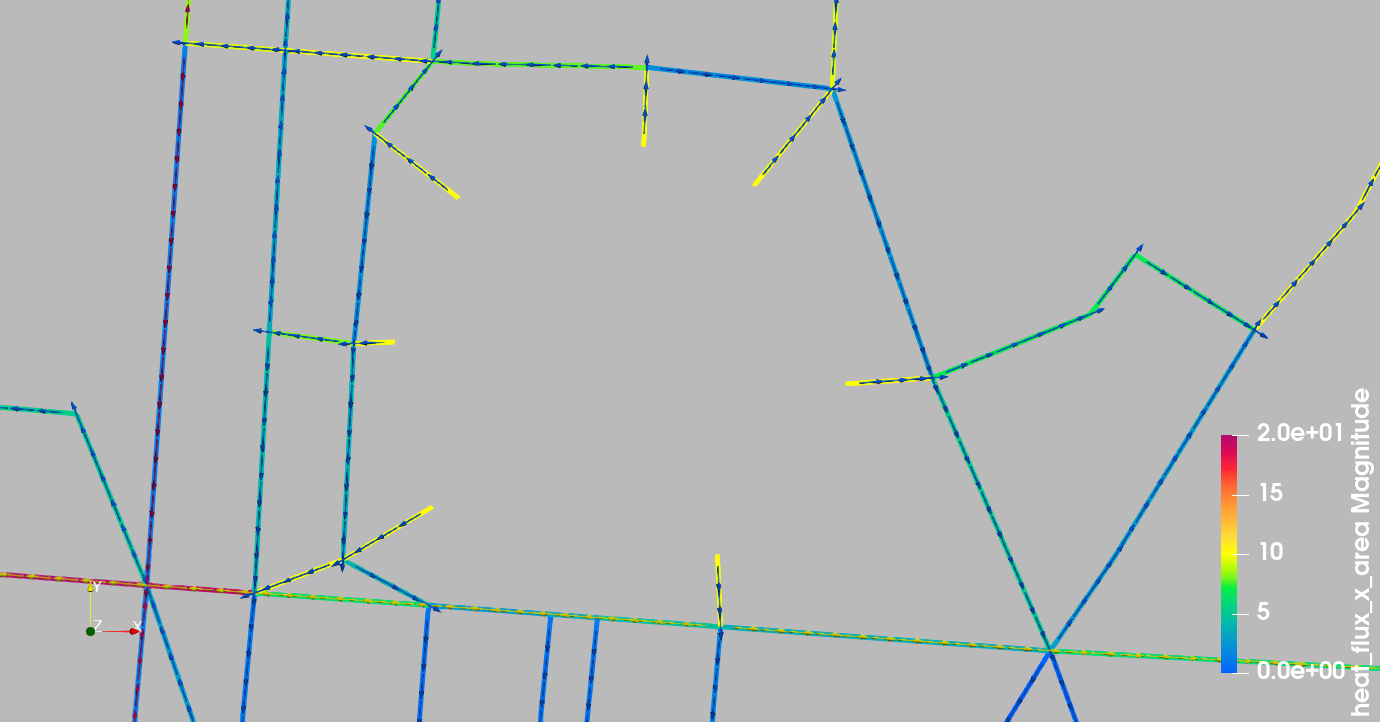}
\caption{Flux throughput ($|q_e|A_e$) across the network.
Top: global view. Bottom: zoom near the central concourse.
Broader corridors support proportionally higher total flux.}
\label{fig:bernabeu-flxxarea}
\end{figure}

\subsubsection{Validation and observations}

Figure~\ref{fig:bernabeu-overview} summarizes the overall network geometry.
The LP formulation yields stable, physically consistent 
solutions even on complex real-world networks.  
No violation of the sign constraints was observed, and total inflow/outflow
fluxes balanced to within $10^{-12}$.  
The major flux paths align with the architectural concourse design,
validating the model’s geometric sensitivity and predictive realism.

\begin{figure}[h!]
\centering
\includegraphics[width=0.8\textwidth]{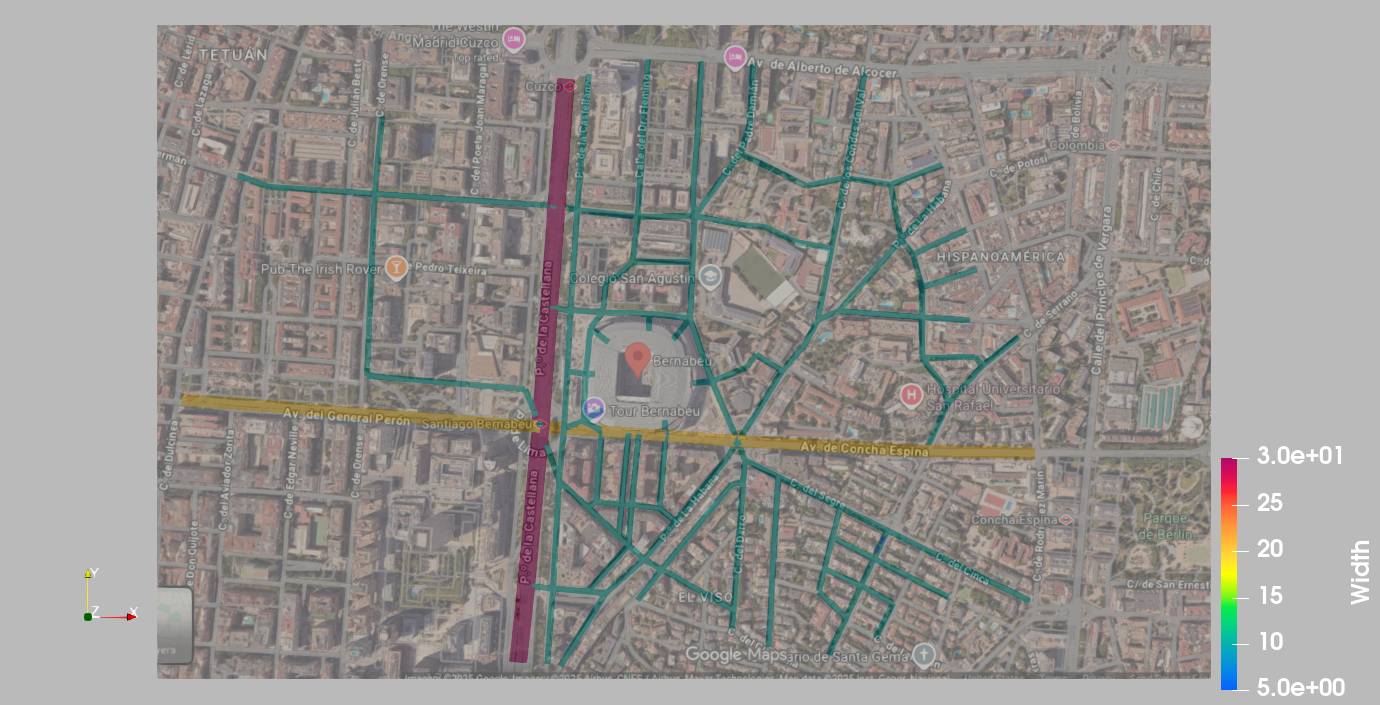}
\caption{Overall stadium network geometry used in the simulation,
overlaid on the planar embedding extracted}.
\label{fig:bernabeu-overview}
\end{figure}

\begin{table}[h!]
\centering
\setlength{\tabcolsep}{8pt}
\renewcommand{\arraystretch}{1.15}
\begin{tabular}{l r}
\toprule
\textbf{Quantity} & \textbf{Value} \\
\midrule
$\max(\Phi_{\text{in}})$ (should $\le 0$)                    & $-1.000\times 10^{+1}$ \\
$\min(\Phi_{\text{out}})$ (should $\ge 0$)                   & $-8.903\times 10^{-15}$ \\
Max boundary–edge sign violation (should $\le 0$)            & $8.90\times 10^{-15}$ \\
Global conservation $\sum_v \Phi_v$ (should $\approx 0$)     & $4.44\times 10^{-14}$ \\
$\max_{v\in\mathcal V_{\text{int}}} |\Phi_v|$                 & $3.936\times 10^{-13}$ \\
Amount entering at $\mathcal V_{\text{in}}$                  & $7.000000\times 10^{+1}$ \\
Amount leaving at $\mathcal V_{\text{out}}$                  & $7.000000\times 10^{+1}$ \\
In–out mismatch (should $\approx 0$)                         & $-1.236\times 10^{-12}$ \\
$\max_{\text{component}} \big|\sum_{v\in \text{comp}} \Phi_v\big|$ & $9.599\times 10^{-15}$ \\
\bottomrule
\end{tabular}
\caption{Stadium exit validation metrics.  
All diagnostics confirm sign correctness and flux conservation 
to within double-precision tolerance ($\sim10^{-14}$).}
\label{tab:bernabeu-metrics}
\end{table}

\paragraph{Key remarks.}
The feasible set $P$ is convex and bounded, ensuring LP solvability. The diffusion-like flux pattern shows that the optimized control redistributes flow to prevent local congestion. The area scaling reveals near-constant per-area flux intensity, confirming that the optimization respects both geometric and capacity constraints.

\subsection{Complex Street Network Emanating from Large, Historical City Center}
\label{sec:kaaba}

We next apply the same LP framework to a geometric network
typical of those encountered around large, historical city centers.
The graph $\calG=(\calV,\calE)$ represents the pedestrian and structural corridors
surrounding the historical center (whose network in turn may
be very complex and chaotic).
Unlike the previous stadium network, this geometry comprises
seven disconnected components of varying scale.
Each component was analyzed independently under identical flux-cap and sign constraints.

\subsubsection{Boundary and reference conditions}

Dirichlet boundary conditions were imposed on inflow (entrance) and outflow (exit) nodes
located at the periphery of each component.
To remove the additive nullspace within each disconnected subgraph,
one node per component was fixed to a prescribed potential:
\[
\text{Nodes } \{10,\,88,\,437,\,152,\,441,\,542,\,550\}, \qquad
u_i = 10.
\]
This ensures uniqueness of the steady-state potential on each connected component
and eliminates the corresponding control-space gauge.

\subsubsection{Temperature (potential) distribution}

The resulting potential field is shown in Figure~\ref{fig:kaaba-temp}.
High potentials (red–yellow) appear near inflow zones, while cooler regions
(blue–green) correspond to outflow boundaries.
Smooth gradients are observed across all seven components,
reflecting physically consistent equilibrium states.

\begin{figure}[h!]
\centering
\includegraphics[width=0.8\textwidth]{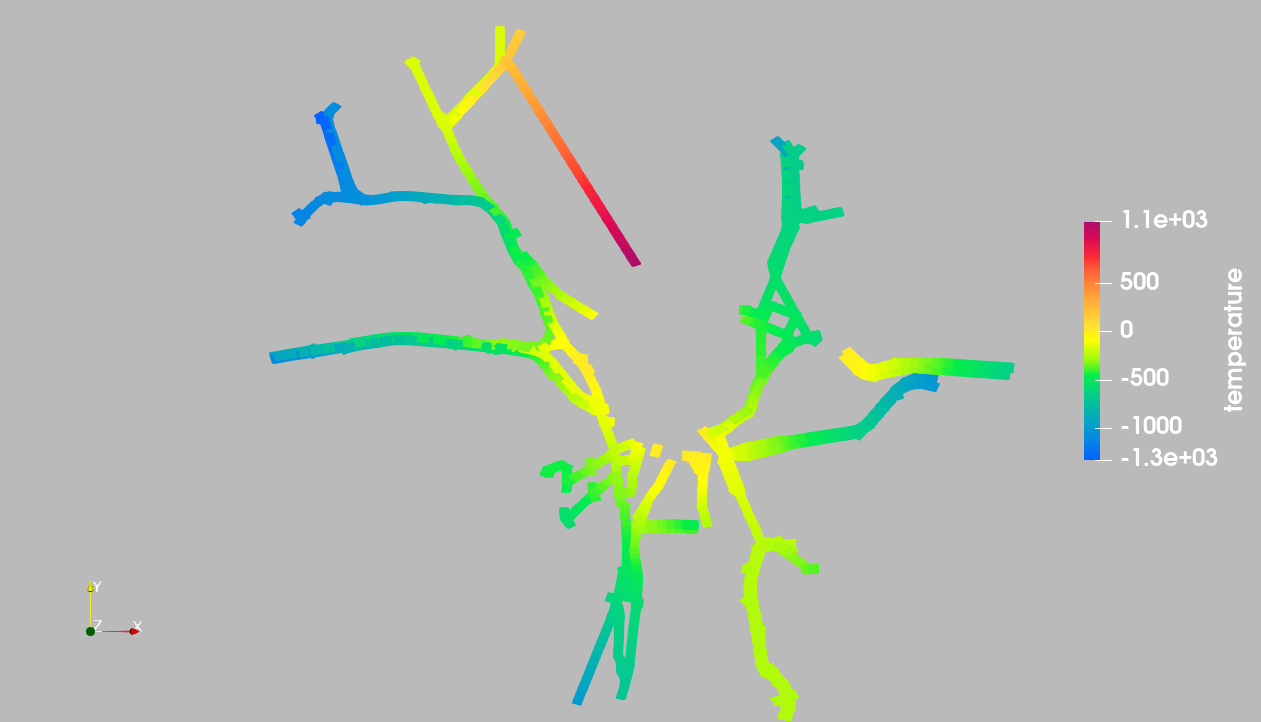}
\caption{Optimized potential field $u$ over the network.
Each connected component exhibits smooth gradients between inflow and outflow boundaries.
Fixed nodes ($u_i=10$) provide a reference potential per component.}
\label{fig:kaaba-temp}
\end{figure}

\subsubsection{Flux intensity, direction, and throughput}

Figure~\ref{fig:kaaba-flux} shows the computed flux intensity $|q_e|$
over the full network, followed by directional and zoomed views
(Figures~\ref{fig:kaaba-flux1}–\ref{fig:kaaba-flux3}).
The LP enforcement of sign constraints yields consistent flux directions,
while stronger magnitudes concentrate along primary corridors.
Figure~\ref{fig:kaaba-hflux-area} presents the total throughput $|q_e|A_e$,
highlighting wide corridors with greater transport capacity.

\begin{figure}[h!]
\centering
\includegraphics[width=0.8\textwidth]{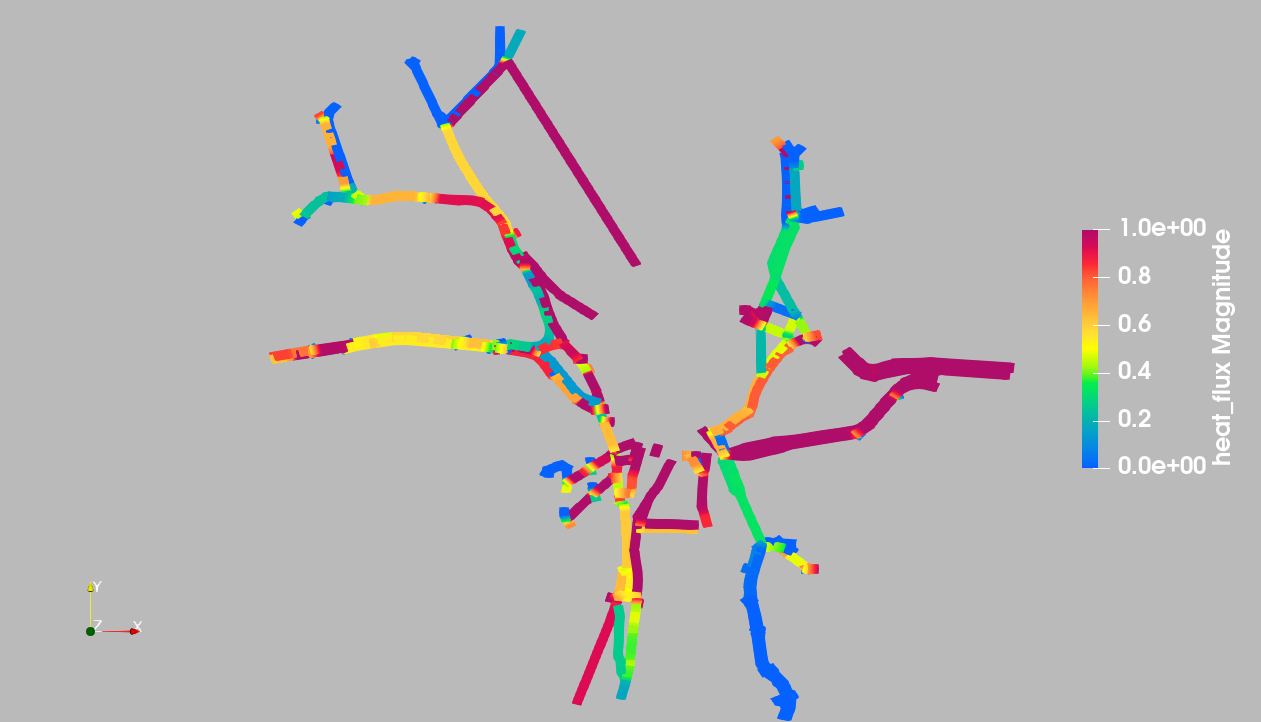}
\caption{Flux intensity $|q_e|$ across the network.
Color scale indicates per-area flux density (blue: weak, red: strong).}
\label{fig:kaaba-flux}
\end{figure}

\begin{figure}[h!]
\centering
\includegraphics[width=0.8\textwidth]{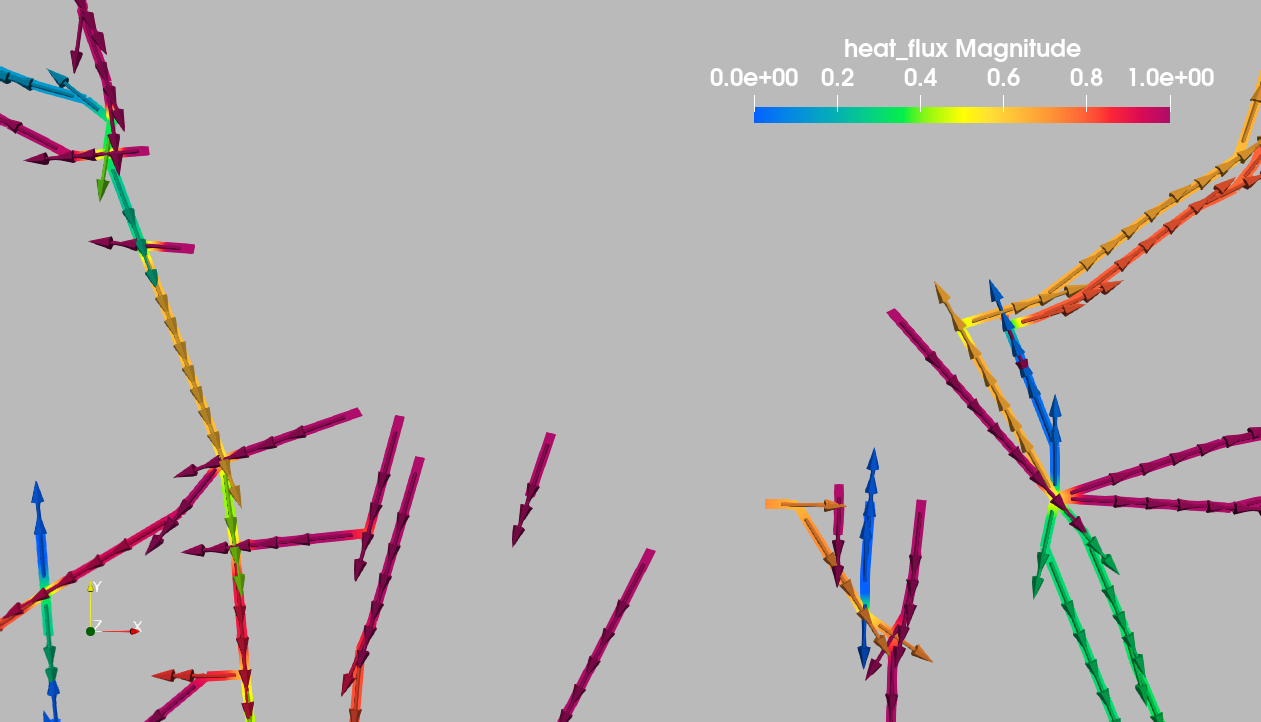}
\vspace{1em}
\includegraphics[width=0.8\textwidth]{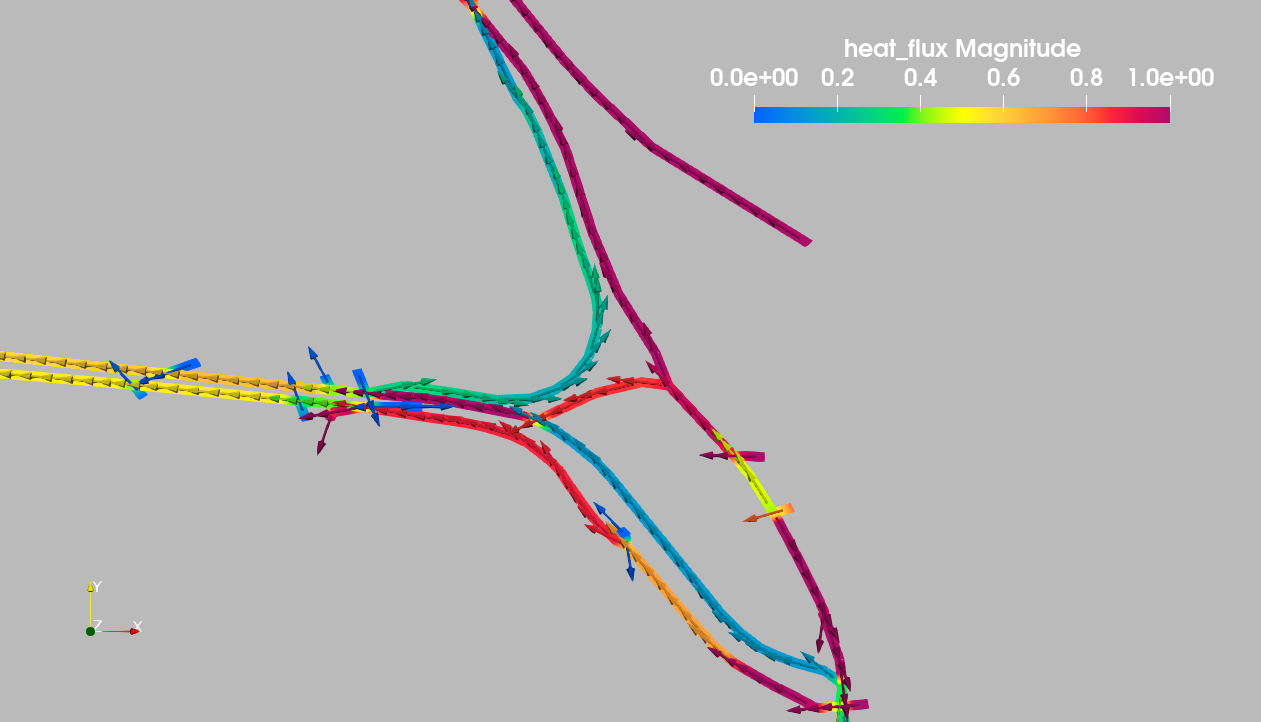}
\caption{Flux directions in two zoomed regions.
Arrowheads indicate consistent directional flow along major corridors.}
\label{fig:kaaba-flux1}
\end{figure}

\begin{figure}[h!]
\centering
\includegraphics[width=0.8\textwidth]{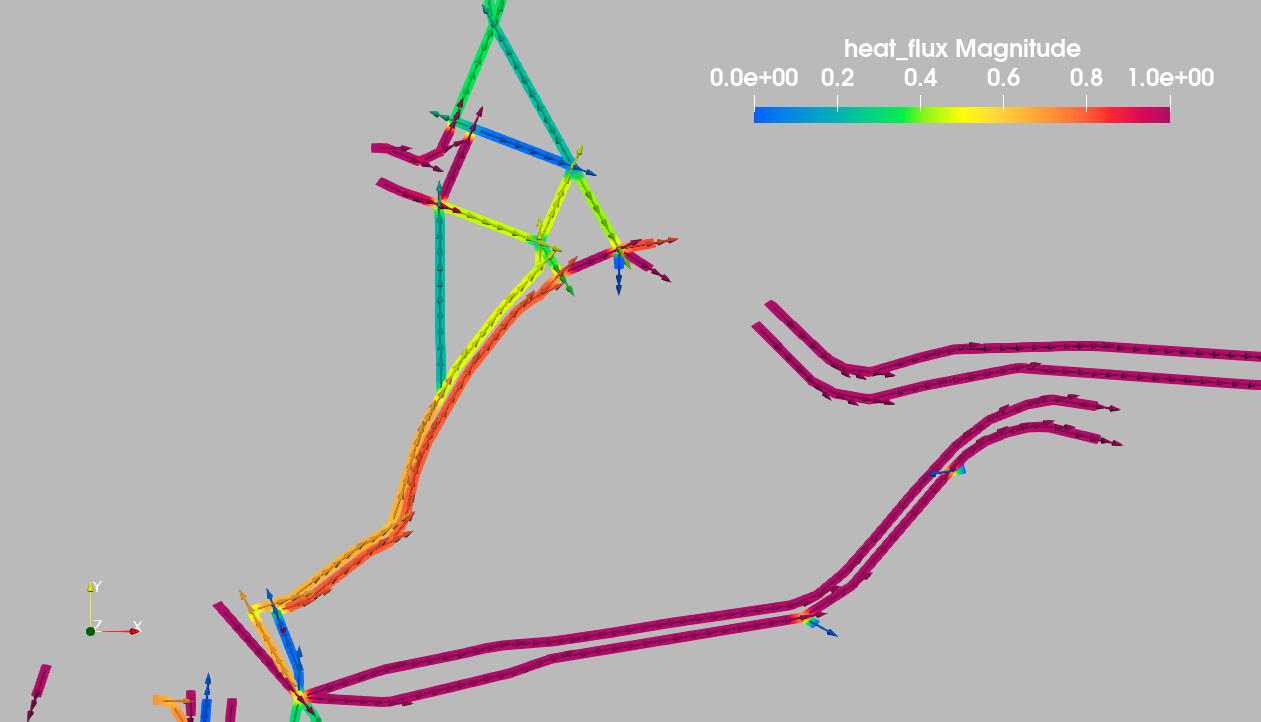}
\caption{Detailed flux orientation near junctions and intersections.}
\label{fig:kaaba-flux3}
\end{figure}

\begin{figure}[h!]
\centering
\includegraphics[width=0.8\textwidth]{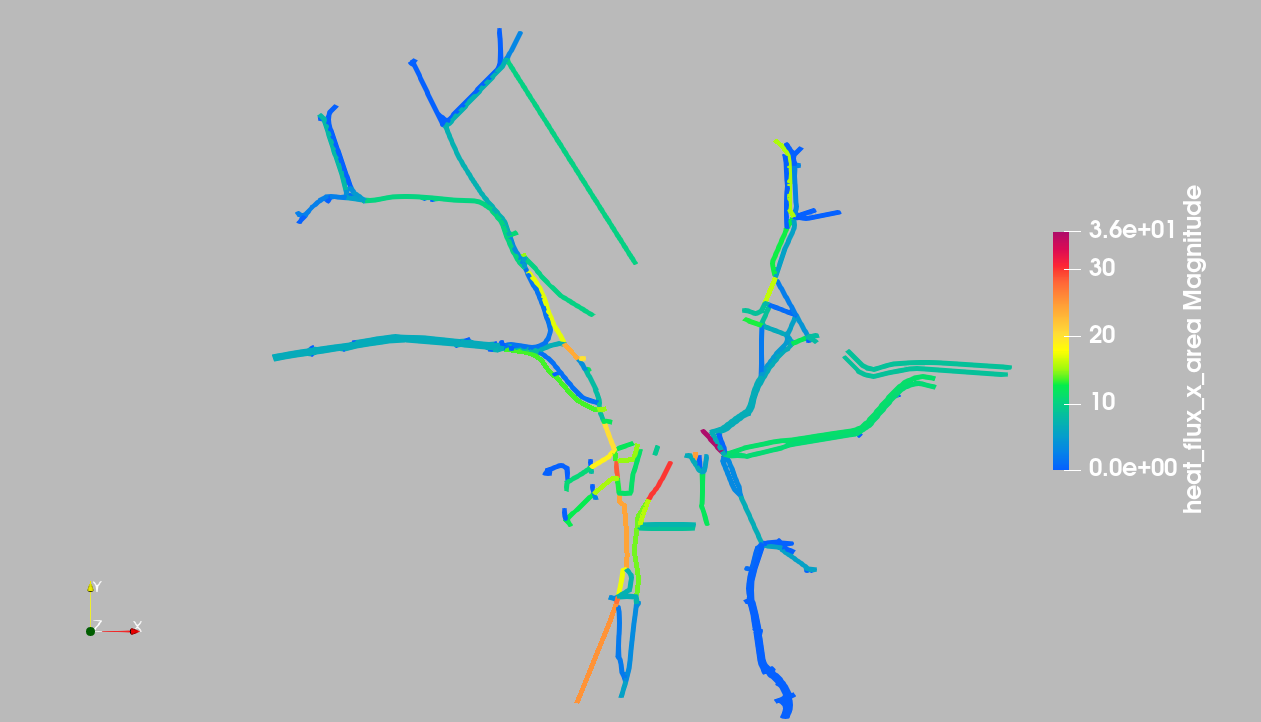}
\caption{Flux throughput ($|q_e|A_e$) across the network.
Wide corridors support proportionally higher total flux, consistent with capacity limits.}
\label{fig:kaaba-hflux-area}
\end{figure}

\subsubsection{Validation and diagnostic metrics}

All conservation and sign checks confirm the correctness of the LP solution.
Table~\ref{tab:kaaba-metrics} summarizes the diagnostic results.
Flux conservation holds to within $10^{-13}$, and boundary sign violations
are below $10^{-6}$, demonstrating consistent performance across multiple components.

\begin{table}[h!]
\centering
\setlength{\tabcolsep}{8pt}
\renewcommand{\arraystretch}{1.15}
\begin{tabular}{l r}
\toprule
\textbf{Quantity} & \textbf{Value} \\
\midrule
$\max(\Phi_{\text{in}})$ (should $\le 0$)                    & $-1.998\times10^{-15}$ \\
$\min(\Phi_{\text{out}})$ (should $\ge 0$)                   & $-3.492\times10^{-6}$ \\
Max boundary–edge sign violation (should $\le 0$)            & $3.49\times10^{-6}$ \\
Global conservation $\sum_v \Phi_v$ (should $\approx 0$)     & $1.07\times10^{-13}$ \\
$\max_{v\in\mathcal V_{\text{int}}} |\Phi_v|$                 & $3.37\times10^{-13}$ \\
Amount entering at $\mathcal V_{\text{in}}$                  & $2.729882\times10^{+2}$ \\
Amount leaving at $\mathcal V_{\text{out}}$                  & $2.729882\times10^{+2}$ \\
In–out mismatch (should $\approx 0$)                         & $4.547\times10^{-13}$ \\
$\max_{\text{component}} \big|\sum_{v\in \text{comp}} \Phi_v\big|$ & $6.682\times10^{-14}$ \\
\bottomrule
\end{tabular}
\caption{Network validation metrics. 
All diagnostics confirm flux conservation and sign correctness
to within numerical tolerance. 
}
\label{tab:kaaba-metrics}
\end{table}

\paragraph{Key remarks.}
The network decomposes into seven connected components, each internally flux-balanced. Independent gauge fixing ($u_i=10$ per component) removes null directions in the control space, ensuring well-posedness. Flux balance and conservation hold to machine precision ($10^{-13}$), confirming numerical robustness. 

These results demonstrate that the proposed framework extends naturally
to multiply connected or fragmented networks while preserving
physical and numerical consistency.

\section*{Conclusion and Outlook}
\label{sec:conclusion}

We have presented a linear programming (LP) framework for steady-state diffusion and flux optimization
on geometric networks.  
Starting from a discrete Laplacian model, we showed that boundary potentials can be treated as control variables
that drive interior states and fluxes through an affine mapping.  
This reduction yields a finite-dimensional LP whose feasible set is polyhedral, whose convexity is guaranteed,
and whose global solvability follows from geometric arguments on recession cones and the
Minkowski–Weyl decomposition.  
Three classes of sufficient conditions---finite box bounds, two-sided flux caps with full-rank maps, and
rank-deficient configurations stabilized by sign constraints---ensure boundedness automatically.

\medskip
Numerical experiments on two real geometric networks, demonstrate the robustness and interpretability of the method.
The LP formulation enforces flux conservation and sign correctness, 
handles disconnected components without additional regularization,
and offers physically meaningful insights into crowd or transport behavior.

\medskip
Beyond its immediate applications to diffusive transport, the proposed framework provides
a foundational building block for \emph{network-based digital twins}.
The affine control-to-flux map enables:
\begin{itemize}
  \item rapid re-optimization of boundary conditions under changing environments;
  \item direct integration into higher-level PDE-constrained or multi-scale optimization loops;  
  \item uncertainty quantification through stochastic or robust LP formulations.
\end{itemize}

\medskip
Several research directions follow naturally.
A first extension concerns time-dependent diffusion or transient transport on networks,
leading to dynamic LPs or model-predictive formulations.
Second, coupling this graph-level LP with continuum PDEs at finer scales would enable
hybrid \emph{graph–PDE digital twins} for structures.

\appendix

\section{Algorithms (pseudocode)}

\begin{algorithm}[H]
\caption{Mixed-boundary outward-flux maximization (high-level)}
\label{alg:driver}
\begin{algorithmic}[1]
\Require Data file; user choice $(\Vfix,u_{\mathrm{fix}})$; parameters $\phi_{\max}$, $\varepsilon$
\State $(\matB,\matC,\matL,\Vin,\Vout)\gets \Call{BuildNetwork}{\text{data}}$     \Comment{incidence, conductance, Laplacian}
\State $(\Vctrl,\Vint,\Efix,\Ectrl,\Eint)\gets \Call{PartitionAndEmbeddings}{\Vfix,\Vin,\Vout}$
\State $(\uzero,\Umap,\qzero,\Qmap,\Phizero,\Pmap,K_{\mathrm{in}},K_{\mathrm{out}},\Phi_{0,\mathrm{in}},\Phi_{0,\mathrm{out}})
    \gets \Call{BuildAffineMaps}{\matL,\matB,\matC,\Efix,\Ectrl,\Eint,u_{\mathrm{fix}}}$
\State $\vecc^\star \gets \Call{BuildObjective}{K_{\mathrm{in}},K_{\mathrm{out}}}$     \Comment{$J=\text{const}+(\vecc^\star)^\top\vecg$}
\State $(A,b,\vecg_{\min},\vecg_{\max}) \gets \Call{AssembleConstraints}{\Qmap,\qzero,\Pmap,\Phizero, \Vin,\Vout, \phi_{\max},\varepsilon}$
\State $\vecg^\star \gets \Call{SolveLP}{-(\vecc^\star),A,b,\vecg_{\min},\vecg_{\max}}$ \Comment{\texttt{linprog}}
\State $(\vecu^\star,\vecq^\star,\vecPhi^\star)\gets \Call{RecoverState}{\uzero,\Umap,\qzero,\Qmap,\Phizero,\Pmap,\vecg^\star}$
\State \Call{VerifyAndReport}{$\Vin,\Vout,\varepsilon,\vecq^\star,\vecPhi^\star$} \Comment{node signs, edge signs, conservation}
\State \Return $(\vecg^\star,\vecu^\star,\vecq^\star,\vecPhi^\star)$
\end{algorithmic}
\end{algorithm}

\begin{algorithm}[H]
\caption{BuildNetwork}
\label{alg:build}
\begin{algorithmic}[1]
\Require Data file with vertices, edges, raw $k_e$, lengths $L_e$, boundary sets $\Vin,\Vout$
\State $c_e \gets k_e/L_e$; \quad $\matC \gets \diag(c_e)$
\State Build edge-by-node incidence $\matB$ with $-1$ at tail, $+1$ at head
\State $\matL \gets \matB^\top \matC \matB$
\State \Return $(\matB,\matC,\matL,\Vin,\Vout)$
\end{algorithmic}
\end{algorithm}

\begin{algorithm}[H]
\caption{PartitionAndEmbeddings}
\label{alg:part}
\begin{algorithmic}[1]
\Require $\Vfix,\Vin,\Vout$; $n_V$ (number of nodes)
\State $\Vctrl \gets (\Vin\cup\Vout)\setminus \Vfix$; \quad $\Vint \gets \{1{:}n_V\}\setminus(\Vfix\cup\Vctrl)$
\State Build embeddings $\Efix,\Ectrl,\Eint$ as sparse placement matrices
\State \Return $(\Vctrl,\Vint,\Efix,\Ectrl,\Eint)$
\end{algorithmic}
\end{algorithm}

\begin{algorithm}[H]
\caption{BuildAffineMaps}
\label{alg:maps}
\begin{algorithmic}[1]
\Require $\matL,\matB,\matC,\Efix,\Ectrl,\Eint,u_{\mathrm{fix}}$
\State $\matL_{ii}\gets \Eint^\top\matL\Eint$;\; $\matL_{i,\mathrm{fix}}\gets \Eint^\top\matL\Efix$;\; $\matL_{i,\mathrm{ctrl}}\gets \Eint^\top\matL\Ectrl$
\State $A_0 \gets -\matL_{ii}^{-1}\matL_{i,\mathrm{fix}}\,u_{\mathrm{fix}}$;\; $A_1 \gets -\matL_{ii}^{-1}\matL_{i,\mathrm{ctrl}}$
\State $\uzero \gets \Eint A_0 + \Efix u_{\mathrm{fix}}$;\; $\Umap \gets \Eint A_1 + \Ectrl$
\State $\qzero \gets -\matC\matB\,\uzero$;\; $\Qmap \gets -\matC\matB\,\Umap$
\State $\Phizero \gets \matB^\top\qzero$;\; $\Pmap \gets \matB^\top\Qmap$
\State $K_{\mathrm{in}} \gets \Pmap(\Vin,:)$;\; $K_{\mathrm{out}} \gets \Pmap(\Vout,:)$
\State $\Phi_{0,\mathrm{in}} \gets \Phizero|_{\Vin}$;\; $\Phi_{0,\mathrm{out}} \gets \Phizero|_{\Vout}$
\State \Return $(\uzero,\Umap,\qzero,\Qmap,\Phizero,\Pmap,K_{\mathrm{in}},K_{\mathrm{out}},\Phi_{0,\mathrm{in}},\Phi_{0,\mathrm{out}})$
\end{algorithmic}
\end{algorithm}

\begin{algorithm}[H]
\caption{BuildObjective}
\label{alg:obj}
\begin{algorithmic}[1]
\Require $K_{\mathrm{in}},K_{\mathrm{out}}$
\State $\vecc^\star \gets \begin{bmatrix}K_{\mathrm{in}}\\ K_{\mathrm{out}}\end{bmatrix}^\top \begin{bmatrix}-\mathbf{1}\\ \mathbf{1}\end{bmatrix}$
\State \Return $\vecc^\star$
\end{algorithmic}
\end{algorithm}

\begin{algorithm}[H]
\caption{AssembleConstraints}
\label{alg:cons}
\begin{algorithmic}[1]
\Require $\Qmap,\qzero,\Pmap,\Phizero,\Vin,\Vout,\phi_{\max},\varepsilon$
\State \textit{Flux caps:} $\vecq_{\max}\gets \phi_{\max}\,\vec k$ 
\State $A_{\mathrm{cap}}\gets \begin{bmatrix}\Qmap\\ -\Qmap\end{bmatrix}$;\;
       $b_{\mathrm{cap}}\gets \begin{bmatrix}\vecq_{\max}-\qzero\\ \vecq_{\max}+\qzero\end{bmatrix}$
\State \textit{Edge no-backflow:} $(A_{\mathrm{edge}},b_{\mathrm{edge}})\gets \Call{EdgeSignRows}{\Vin,\Vout,\qzero,\Qmap,\varepsilon}$
\State $A\gets \begin{bmatrix} A_{\mathrm{cap}}\\ A_{\mathrm{edge}}\end{bmatrix}$;\;
       $b\gets \begin{bmatrix} b_{\mathrm{cap}}\\ b_{\mathrm{edge}}\end{bmatrix}$
\State Choose bounds $\vecg_{\min},\vecg_{\max}$ (optional)
\State \Return $(A,b,\vecg_{\min},\vecg_{\max})$
\end{algorithmic}
\end{algorithm}

\begin{algorithm}[H]
\caption{EdgeSignRows (build edge-level no-backflow)}
\label{alg:edge}
\begin{algorithmic}[1]
\Require Graph orientation (tails/heads via $\matB$ indices), sets $\Vin,\Vout$, affine $(\qzero,\Qmap)$, slack $\varepsilon$
\State Build a sparse selector $S$ with one row per boundary endpoint rule:
\Statex \quad if $t\in\Vin$, add row picking $+q_e$;\; if $h\in\Vin$, add row picking $-q_e$;
\Statex \quad if $h\in\Vout$, add row picking $+q_e$;\; if $t\in\Vout$, add row picking $-q_e$
\State \Return $A_{\mathrm{edge}} \gets -S\,\Qmap$;\quad $b_{\mathrm{edge}} \gets S\,\qzero + \varepsilon$
\end{algorithmic}
\end{algorithm}


\begin{algorithm}[H]
\caption{RecoverState and VerifyAndReport}
\label{alg:recover-verify}
\begin{algorithmic}[1]
\Require $(\uzero,\Umap,\qzero,\Qmap,\Phizero,\Pmap,\vecg^\star)$, sets $(\Vin,\Vout)$, slack $\varepsilon$
\State \textbf{Recover:} $\vecu^\star \gets \uzero+\Umap\vecg^\star$;\; $\vecq^\star \gets \qzero+\Qmap\vecg^\star$;\; $\vecPhi^\star \gets \Phizero+\Pmap\vecg^\star$
\State \textbf{Node signs:} ensure $\max(\vecPhi^\star|_{\Vin})\le 0$, $\min(\vecPhi^\star|_{\Vout})\ge 0$
\State \textbf{Edge signs:} rebuild $S$; check $\min(S\vecq^\star)\ge -\varepsilon$; list any violators
\State \textbf{Conservation:} check $\abs{\sum_v \Phi^\star_v}\approx 0$
\State \Return $(\vecu^\star,\vecq^\star,\vecPhi^\star)$
\end{algorithmic}
\end{algorithm}

\bibliographystyle{plain}
\bibliography{refs}

\end{document}